\documentclass[12pt]{article}

\usepackage{relsize}
\usepackage{fullpage}
\usepackage{epsfig}
\usepackage{amsfonts}
\usepackage{amssymb}
\usepackage{amsmath}   
\usepackage{amsthm}
\usepackage{latexsym}
\usepackage[square,sort,comma,numbers]{natbib}

\usepackage{csquotes}
\usepackage[shortlabels]{enumitem}
\usepackage{tikz}
\usepackage{pgfplots}
\usepackage{etoolbox}
\usepackage[utf8]{inputenc}
\usepackage{caption}
\usepackage{subcaption}

\usetikzlibrary{decorations.markings}
\pgfplotsset{compat=1.8}

\usetikzlibrary{calc,patterns}

\newtheorem{theorem}{Theorem}[section]
\newtheorem{lemma}[theorem]{Lemma}
\newtheorem{corollary}[theorem]{Corollary}
\newtheorem{proposition}[theorem]{Proposition}
\newtheorem{definition}[theorem]{Definition}

\newtheorem{remark}[theorem]{Remark}

\newtheorem*{theorem*}{Theorem}

\newcommand{\of}[1]{\left( #1 \right)}
\newcommand{\beqn}{\begin{equation}}
\newcommand{\eeqn}{\end{equation}}
\newcommand{\norm}[1]{\left\Vert #1 \right\Vert}
\newcommand{\abs}[1]{\left\vert #1 \right\vert}
\newcommand{\R}{\mathbb{R}}
\newcommand{\C}{\mathbb{C}}
\newcommand{\N}{\mathbb{N}}
\newcommand{\D}{\mathbb{D}}

\newcommand{\im}{\textrm{im }}
\newcommand{\id}{\textrm{Id}}

\newcommand{\Strip}{\textrm{St}}
\renewcommand{\Re}{\textrm{Re}}
\renewcommand{\Im}{\textrm{Im}}

\newtoggle{draw_figures}
\toggletrue{draw_figures}

\newtoggle{compile_questions}
\togglefalse{compile_questions}

\newtoggle{compile_ideas}
\togglefalse{compile_ideas}

\pgfdeclarepatternformonly{MyGrid}{
\pgfqpoint{0pt}{0pt}}{\pgfqpoint{8pt}{8pt}}{\pgfqpoint{7pt}{7pt}}
{
  \pgfsetlinewidth{0.6pt}
  \pgfpathmoveto{\pgfqpoint{0pt}{0pt}}
  \pgfpathlineto{\pgfqpoint{8pt}{8pt}} 
  \pgfpathmoveto{\pgfqpoint{8pt}{0pt}}
  \pgfpathlineto{\pgfqpoint{0pt}{8pt}} 
  \pgfusepath{stroke}
}

\newdimen\GridSize
\tikzset{
    GridSize/.code={\GridSize=#1},
    GridSize=3pt
}

\pgfdeclarepatternformonly{my north west lines}{%
\pgfqpoint{0pt}{0pt}}{\pgfqpoint{8pt}{8pt}}{\pgfqpoint{7pt}{7pt}}%
{
  \pgfsetlinewidth{0.6pt}
  \pgfpathmoveto{\pgfqpoint{0pt}{0pt}}
  \pgfpathlineto{\pgfqpoint{8pt}{8pt}} 
  \pgfusepath{stroke}
}

\pgfdeclarepatternformonly{my north east lines}{%
\pgfqpoint{0pt}{0pt}}{\pgfqpoint{8pt}{8pt}}{\pgfqpoint{7pt}{7pt}}%
{
  \pgfsetlinewidth{0.6pt}
  \pgfpathmoveto{\pgfqpoint{0pt}{8pt}}
  \pgfpathlineto{\pgfqpoint{8pt}{0pt}} 
  \pgfusepath{stroke}
}

\date{}

\title{Complex time blow-up of the nonlinear heat equation}

\author{    \textbf{Hannes Stuke} \footnote{Free University of Berlin, Berlin, Germany; email: h.stuke@fu-berlin.de} }

\begin{document}

\pagenumbering{gobble}

\maketitle

\section*{Abstract}

This paper investigates the connection between blow-up solutions of scalar reaction-diffusion equations, in particular of $u_t = u_{xx} + u^2,
$ and its counterpart - eternally existing solutions like heteroclinic orbits - by complex time.
We prove that heteroclinic orbits in one-dimensional unstable manifolds are accompanied by blow-up solutions. Furthermore we show, that we can continue blow-up solutions into a slit complex time and eventually back to the real axis. The solution picks up an imaginary factor after continuation which is related to the eigenvalue relations of the linearizations at the source and the sink of the heteroclinic orbit.

\clearpage

\addtocontents{toc}{\protect\setcounter{tocdepth}{1}}


\newpage

\pagenumbering{arabic}


\section{Introduction}

In this paper we study the nonlinear heat equation with Dirichlet boundary conditions

\begin{equation} \label{eq:nonlinear_heat_general_p}
u_t = \Delta u + u^p, \qquad x \in \Omega \subset \R^N, \qquad u \vert_{\partial \Omega}  = 0, \qquad p > 1 .
\end{equation}

It is well-known that solutions may blow-up and much is known about the blow-up behaviour, i.e. blow-up rate, blow-up set and blow-up profile, see \citep{quittner07}, \citep{hu11} and the references therein.

There have also been attempts to study the complex-valued nonlinear heat equation \citep{guo12}, \citep{zaag15} for real time only, as one can under that circumstances rewrite \eqref{eq:nonlinear_heat_quad} as system of real and imaginary part of $ u = v + i w $.

\begin{align*}
v_t &= v_{xx} + v^2 - w^2, \\
w_t &= w_{xx} + 2 w v.
\end{align*}

Especially \citep{zaag15} was able to derive detailed asymptotic expansions of the blow-up profile close to the blow-up. 

A further question asked already by \citep{vazquez04} concerns the continuation of solutions after blow-up. One possibility to tackle the problem is monotone approximation from below, which let to the notion of complete blow-up, see e.g. \citep{baras87}, \citep{lacey88}, \citep{martel98}, \citep{polacik05}, \citep{quittner07}.

Consider the following truncations of the nonlinearity $ f$ defined by $f_k \of{u} := \min \of{ u, k}$. Then the equation to the truncated nonlinearity

\[
u^k_t = \Delta u^k + f_k \of{u^k}, \qquad u^k \of{0} = u_0,
\]

possess global solutions. For any $ k$ the solution $u^k $ coincides with the solution $ u $ of \eqref{eq:nonlinear_heat_general_p} as long as $ u $ is below $k$. Suppose now, that $ u $ is positive and blows up at time $ T $. Then the solutions $u^k $ monotonically approximate $ u $ from below before the blow-up and even exist after the blow-up. So the question is now, if the solutions $ u^k $ converge for $ k \to \infty $ and $ t > T $, that is in which sense the function

\[
\bar u \of{t,x} := \lim_{k\to \infty} u^k \of{t, x}, 
\]

exists.

\vspace*{\fill}

\paragraph*{Acknowledgements} The author wants to thank Bernold Fielder, Pavel Gurevich, Jia-Yuan Dai and Phillipo Lappicy for inspiring discussions. The author was supported by SFB 647 and SFB 910.

\newpage

The time of complete blow-up is defined as follows

\[
T_c := \inf \left \{ t \geq T, \; \bar u \of{t,x} = \infty, \; \forall x \in \Omega \right \}.
\]

Note, that after $ T_c $, the solution is unbounded for all $ x \in \Omega $. This implies that there are no \enquote{weak} solutions that are compatible with point wise approximations from below, e.g. solutions of weaker integrability or measure-valued solutions after blow-up. A solution blows-up completely at $ T $ if $ T_c = T $. 
The quadratic one-dimensional nonlinear heat equation has only complete blow-up \citep{quittner07}.

In this paper we will study the blow-up behaviour from the perspective of complex time. We will especially tackle the problem of continuation, since in the situation of complete blow-up we are lacking a proper notion of continued solutions. We consider the quadratic nonlinear heat equation

\begin{equation} \label{eq:nonlinear_heat_eq_intro}
u_t = u_{xx} + u^2, \qquad x \in \left(-1, 1 \right), \qquad u \of{\pm 1} = 0.
\end{equation}

and try to address the following questions

\begin{enumerate}[label=(\roman*)]
\item Can we prove the existence of blow-up by studying only bounded solutions?
\item Can we extend the solution after blow-up though complex time back to the real axis? And do analytic continuations coincide after real axis along different paths of continuation?
\end{enumerate}

In Section \ref{sec:unstable_manifold_boundary_het_blow_up} we prove that if there exists a one-dimensional fast unstable manifold of an evolution equation of the form

\[
u_t = A u + f \of{u},
\] 

which contains a heteroclinic orbit connecting two equilibria, there must be a complex-time blow-up orbit starting with real initial data and a real time blow-up orbit starting with complex initial data. \\ 
We give a more detailed description for blow-up solutions of equation \eqref{eq:nonlinear_heat_eq_intro}. Equation \eqref{eq:nonlinear_heat_eq_intro} possesses a positive equilibrium $ u_+ $ with one-dimensional fast unstable manifold $ W^u $. Locally the manifold is a graph $ \Upsilon $ over the eigenspace of the fast unstable mode. One side of the fast unstable manifold is a heteroclinic orbit to zero, whereas the other side is a blow up orbit, \citep{quittner07}. 

We denote the solution flow of equation \eqref{eq:nonlinear_heat_eq_intro} by $ \Phi(t,u_0) $. The domain to which we can continue solutions are so called spall strips.

\begin{definition}[Spall strip]
For any $ \delta > 0 $ and $ 0 < T < T_1 $ we denote the upper/lower spall strip as
\[
S_\pm \of{\delta, \left[ T, T_1 \right] } := \left \{ t \in \C \setminus \left[T, T_1 \right], \; 0 \leq \pm \Im \of{t} \leq \delta  \right \}.
\]
\end{definition}

\begin{figure}[h]
\centering
\begin{tikzpicture}

\draw[fill=white, draw=white, postaction={pattern=my north west lines,pattern color=gray}] (-1,1) -- (-1,0) -- (4,0)  -- (4,1) -- (-1,1);

\draw  (4,1) -- (-1,1) node[left] {$\delta$};
\draw (0,1) -- (1,1);
\draw[fill] (0,0) circle [radius = 0.07cm] node[below] {$0$};

\draw (-1,0) -- (1.43,0);
\draw (3.07,0) -- (4,0);

\draw[line width = 0.15 cm, white] (1.5,0) -- (3,0);

\draw (1.5,0) circle [radius = 0.07cm] node[below] { $ T $};

\draw (3,0) circle [radius = 0.07cm] node[below] { $ T_1 $};  

\end{tikzpicture}
\caption{Upper Spall strip domain}
\end{figure}
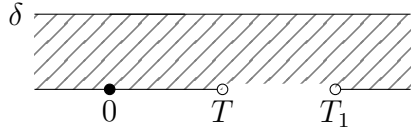

The time $T_1 $ is also called resurrection time.

We will prove the following theorem.

\begin{theorem}
There exists a $ \delta > 0 $ such that the time analytic continuations of the real blow up orbit on the fast unstable manifold of $u_+ $, i.e. $ \Phi \of{t, (\tau, \Upsilon \of{\tau}} $, $ 0 < \tau < \delta $ exists and has the following properties:
\begin{enumerate}[label=(\roman*)]
\item It blows up completely after at time $ T$.
\item It can be continued to upper and lower spall strips $S_\pm \of{\delta, \left[ T, T_1 \right] }$  for some
\[
T_1 < 2 C_0 \max \left\{ \max_{x \in I} \frac{v_0 \of{x}}{w_0 \of{x}}, \norm{u_+}_\infty \right \}.
\]
The constant $ C_0 > 0 $ depends only on the heat semigroup.
\item The upper and lower time path continuations do not coincide after $ T_1 > 0 $
\end{enumerate}
\end{theorem}

We furthermore show that the branch-type of the complex-time continuation is related to the quotient of the eigenvalues of the eigenfunctions to which the heteroclinic orbit is tangent at $ u_+ $ resp at $ u \equiv 0 $. 

There have already been previous attempts to the question of analytic time continuations, mainly \citep{guo12} and \citep{masuda82}, \citep{masuda84}.

Masuda \citep{masuda82}, \citep{masuda84} has already proven the following results for the quadratic nonlinear heat equation with Neumann boundary conditions

Define the constant $ a $ as follows

\[
a := \frac{1}{2} \int_{-1}^1 u_0 \of{x} dx.
\]

\begin{theorem} \label{thm:thm_masuda1}
Let $u_0 $ be a non-negative function ( $ u_0 \neq 0 $ ) in $W^2_p \of{\Omega} $, $p > n$ , and set $a = P u_0$. If $ \norm{\partial_x^2 u_0}_p / \abs{a}^2 $ is sufficiently small, then there exists a unique solution $u_j$
( $j = 1, 2$ ) of the quadratic nonlinear heat equation with Neumann boundary conditions which is analytic in $t \in D_j$  as a $W^2_p \of{\Omega} $-valued function and converges to $u_0$ as $t \to 0$, $ \abs{\arg t} < \theta$ , in the norm of $W^2_p \of{\Omega} $.
\end{theorem}

This result shows the existence of analytic continuations for solutions with almost constant initial data in the regions $ D_{1,2} $. Note, that $ D_1 = \bar{D}_2 $, see Figure \ref{fig:continuation_masuda}.

\begin{figure}[h]
\centering
\begin{tikzpicture}

\draw[->] (0,-2) -- (0,2) node[right] {$\Im \of{ t}$};
\draw[->] (-1,0) -- (5,0) node[below] {$\Re \of{t}$};

\draw (2,2) -- (0,0) -- (0.5,-0.5) -- (1.5,0.5) -- (3,-1);

\draw[fill] (1.5,0) circle (0.02cm) node[below] {$t_0$};

\draw (3,1) node {$D_1$};

\end{tikzpicture}
\caption{Existence of solution in complex time} \label{fig:continuation_masuda}
\end{figure}
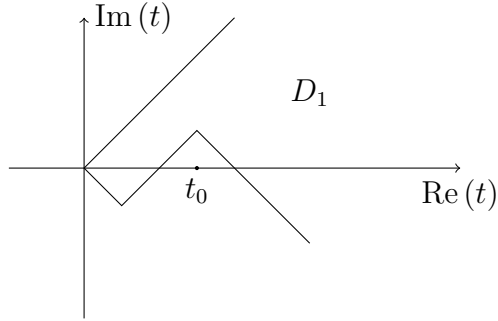

\begin{theorem}
Let $u_j$ , $j=1, 2$ be as in Theorem \ref{thm:thm_masuda1}. If $ \norm{\partial_x^2 u_0}_p / \abs{a}^2 $ is sufficiently small and $u_1 \of{t, \cdot} = u_2 \of{t, \cdot} $ for some $ t \in D_1 \cap D_2 $ for $ \Re \of{t} > t_0 + \delta $, then is $u_0 $  a constant function.
\end{theorem}

The proofs of Masuda rely on the implicit function theorem with respect to the explicit solution to spatially constant initial conditions.\\
Also \citep{guo12} has tackled the problem of the quadratic nonlinear heat equation. They considered the equation for real and imaginary part of $ u = v + i w $ separately, but for real time only

\begin{align} \label{eq:quadratic_system}
\begin{pmatrix} v_t \\ w_t \end{pmatrix}
 = \begin{pmatrix} v_{xx} \\ w_{xx} \end{pmatrix}
  + \begin{pmatrix} v^2 - w^2 \\
2 v w \end{pmatrix}, \qquad  x \in \R.
\end{align}

They observed that the image of the real time flow $ t \in \R \mapsto \eta \of{t, v_0, w_0} $ of the reaction term 

\begin{align} \label{eq:quadratic_system_rd}
\begin{pmatrix} v_t \\ w_t \end{pmatrix}
 =  \begin{pmatrix} v^2 - w^2 \\
2 v w \end{pmatrix},
\end{align}

is a circle in $ \R^2 $ if $ w_0 \neq 0 $. If the image of the ODE flow of the reaction term in a system of reaction-diffusion equations is a convex subset of $ \R^2 $ and the diffusion parts do not couple, then the system possesses a maximum principle \citep{weinberger75}, \citep{evans10}:

If the initial and the boundary conditions of system \eqref{eq:quadratic_system} is contained in the interior of a solution to \eqref{eq:quadratic_system_rd}, then is the solution of \eqref{eq:quadratic_system} contained in the interior for all positive times. \\

Note, that the maximum principle may not hold anymore if the diffusion parts start to couple. This happens if one rotates the time axis into the complex plane, e.g. $ t \mapsto e^{i \theta} t $, which yields the equation 

\begin{align*} 
\begin{pmatrix} v_t \\ w_t \end{pmatrix}
 = \begin{pmatrix}
 \cos \theta & -\sin \theta \\
 \sin \theta & \cos \theta
 \end{pmatrix} \left[ \begin{pmatrix} v_{xx} \\ w_{xx} \end{pmatrix}  + \begin{pmatrix} u^2 - v^2 \\  2 v w\end{pmatrix} \right].
\end{align*}

But if we solve \eqref{eq:nonlinear_heat_eq_intro} along complex time paths that are parallel to the real time axis ($ \theta = 0$), we can always rewrite the equation as a system of real and imaginary part of the form \eqref{eq:quadratic_system}. \\

Using the maximum principle, \citep{guo12} also proved convergence to zero if the initial values $ v_0 \of{x} $, $ w_0 \of{x} $ satisfy 

\[
v_0 \of{x} - B w_0 \of{x} < 0, \qquad \forall x \in \R,
\]
and some $ B \in \R $.

We will use a very similar result to show existence and convergence of complex time analytically continued solutions for times with large real part.

Throughout the chapter we need the following definitions.


\begin{definition}[Strip] \label{def:strip}
We denote by $\Strip_{r_1, r_2, \delta} $ rectangular subsets of the complex plane, that is
\[
\Strip_{r_1, r_2, \delta} := \left \{ x + i y \in \C, \; r_1 < x < r_2, \; \abs{y} < \delta \right \}.
\]
The strip $\Strip_{0, r_2, \delta}$  is abbreviated by $\Strip_{r_2, \delta}$ and the unbounded half-strip $ \Strip_{0, \infty, \delta}$ by $ \Strip_\delta $.
\end{definition}

\begin{definition}[Time $p$--path] \label{def:time-p-path}
 We define a path $ \gamma : I \to \C, t \mapsto t + i \tilde \delta $ for some fixed $ \tilde \delta $ and some interval $I$ as time parallel-path or time $p$--path, since the line $ \gamma $ is parallel to the real axis.
\end{definition}

\newpage

\section{The unstable manifold theorem and its complex consequences} \label{sec:unstable_manifold_boundary_het_blow_up}

We first prove a connection between blow-up solutions and its counterpart, in particular heteroclinic orbits in complex time. Complex time and analyticity of the solutions allow to connect these solutions, because of strong compactness properties and limit theorems for analytic functions, which we will state for convenience.

\begin{theorem}[Vitali theorem, \citep{Arendt_Vector_valued_laplace_transform}] \label{thm:vitali_convergence_thm_c}
Let $ X $ be a Banach space and the family $ f_n : \Omega \to X $ holomorphic and bounded on compact subsets of $ \Omega $, $ \Omega $ open and connected. Assume that the set
\[
\Omega_0 := \left \{ z \in \Omega: \lim_{n \to \infty} f_n \of{z} \textrm{exists} \right \},
\]
has a limit point in $\Omega $. Then there exists a holomorphic function $ f : \Omega \to X $ such that
\[
 f^{(k)} \of{z} = \lim_{n \to \infty}  f^{(k)}_n \of{z}
\] 
uniformly on compact subsets.
\end{theorem}

We will apply Vitali's theorem to the flow of differential equations of which we can describe the $ \omega $-limit set, i.e. the set $ \Omega_0 $. We will first introduce the functional analytic setting.

\begin{equation} \label{eq:canonical_equation}
u_t = Au + f \of{u},
\end{equation}

with $ f \of{0} = 0 $ and $ Df \of{0} = 0$. 
\begin{enumerate}[label=(\roman*)]
\item Consider the complex separable Banach spaces $ Z $, $ Y $ and $ X $ such that the embeddings $ Z \hookrightarrow Y \hookrightarrow X $ are continuous.
\item $A : Z \to X $ is a bounded operator. 
\item $ A $ is a sectorial operator with $ \abs{ \Re \of{ \sigma \of{A}}} \geq \delta > 0 $.
\item $ f \in C^\omega \of{Z, Y} $ and $ f \of{\R} \subset \R$.
\item The operator $ A $ has a compact resolvent.
\end{enumerate}

As canonical example one should think of the quadratic nonlinear heat equation with $ X = L^2 \of{I, \C} $, $ Y = H_0^1 \of{I, \C} $ and $ Z = H^2 \of{I, \C} \cap H^1_0 \of{I, \C}$, $ I = \left(-1, 1 \right) $.

We also assume that there exists a second real-valued  equilibrium $ u_+ $ of the equation \eqref{eq:canonical_equation} whose linearization $A_+ := A + Df \of{u_+} $ has similar properties to $ A$.

The above conditions imply the existence of stable and unstable manifolds around $ u \equiv 0 $ and $ u_+ $, see e.g. \citep{lunardi95}, \citep{henry93} and \citep{vanderbauwhede2}. Since the resolvent of $ A$ is compact, the spectrum is discrete and we can construct strong unstable manifolds. Throughout this chapter, we make the further assumption that the largest eigenvalue of $ A_+$ is algebraically simple with eigenvalue $ \mu > 0 $ and eigenfunction $ \varphi_\mu $. This implies, that the strong unstable manifold associated to $ \mu $ is one-dimensional. We will denote the associated spectral projection $ P_+ $ and $ P_- = \id - P_+ $.

Also by \citep{henry93} we know that the strong unstable manifold is actually analytic, that is the graph of the real strong unstable manifold $ \Upsilon : \left(-a, a \right) \to Z_- $ is an analytic function which can be extended to a complex neighborhood of zero. Here $ Z_- := P_- Z$. The flow on the strong unstable manifold is a one-dimensional analytic differential equation

\[
\dot q = \mu q + P_+ f \of{q \varphi_\mu + \Upsilon \of{q}},
\]

Instead of $ f \of{q \varphi_\mu + \Upsilon \of{q}} $ we will also write $f \of{q,  \Upsilon \of{q}}$.

The solution semigroup of ordinary differential equations can be extended to a group in complex time. This will be important together with the fact that the ordinary differential equation depends analytically on the initial data on the unstable manifold, see for example \citep{hille97}. \\

We begin with a very simple Lemma about the connection of unbounded and bounded solutions for analytic systems. It is a direct corollary to Vitali's theorem \ref{thm:vitali_convergence_thm_c}.

\begin{lemma} \label{lem:unbounded_solutions_omega_limit}
Consider a map $ \Phi : \R_+ \times U \subset \C \to Z $ with the following properties
\begin{enumerate}[label=(\roman*)]
\item $ \Phi \of{t, 0} = 0 $, $ t > 0 $,
\item $ \Phi \of{t, \cdot} \in C^\omega \of{U, Z} $ for $ t > 0 $.
\item $ \Phi \of{t, \cdot} $ is uniformly bounded on compact subsets of $ U $.
\end{enumerate}
$ U \subset \C $ is an open and connected neighbourhood of zero. Assume that there exists an element $ q_0 \in U $ and a sequence of pairwise different $ q_m \in U $ with $ \lim_{m \to \infty} q_m = q_0 \in U $ such that
\[
\lim_{t \to \infty} \Phi \of{t,q_m} \to p^*, \qquad m \in \N \cup \left \{ 0 \right \}, \; \qquad p^* \in Z.
\]
Then for all $ q \in U $ holds,
\[
\lim_{t \to \infty} \Phi \of{t,q} = p^*, \qquad q \in U.
\]
and in particular $ p^* = 0$.
\end{lemma}
\begin{proof}
Consider any sequence $ t_n \nearrow \infty $ and define the family of analytic functions $ f_n $ defined as the time $t_n$ -- maps of the neighborhood $ U $, i.e. 
\[
f_n : U \to Z, \; \qquad f_n \of{u} := \Phi \of{t_n, q}.
\]
By assumption the set $ \Omega_0 $ defined as
\[
\Omega_0 := \left \{ z \in U: \lim_{n \to \infty} f_n \of{z} \textrm{ exists } \right \},
\]
non--empty and has a limit point $ q_0 \in U $. Thus by Vitali's theorem \ref{thm:vitali_convergence_thm_c} we know that the family $ f_n $ converges uniformly to an analytic function $ f : U \to Z $ for all $ \tilde q_0 \in U $. Note, that the limit exists for the full sequence and not just for a subsequence. Here it is important, that the claim does not follow from Montel compactness properties, but uses the analyticity of the functions. By assumption we have $ f \of{q_m} = p^* $ and the sequence $ q_m $ also has an interior limit point. We can conclude that $ f $ is constant by the identity theorem for analytic functions, that is $ f \equiv p^* $. Thus, the full sequence converges to $ p^*$
\[
\lim_{n \to \infty} \Phi \of{t_n, q} = p^*.
\]
Since the sequence $ t_n $ was arbitrary the claim holds for any sequence $ t_n \nearrow \infty $.
\end{proof}

\begin{remark}
Note, that this theorem does not follow from Montel compactness. We get the limit on the full sequence and not just a converging subsequence and furthermore, we do not require any compactness of the image of $ f_n $. The existence of the limit follows from assumptions on the regularity and give stronger convergence results.
\end{remark}

\begin{lemma} \label{lem:complex_time_grow_up_semiflow}
Consider a map $ \Phi : \R_+ \times U \subset \C^N \to Z $, with the following properties
\begin{enumerate}[label=(\roman*)]
\item $ \Phi \of{t, 0} = 0 $, $ t > 0 $,
\item For $ t > 0 $, $ \Phi \of{t, \cdot} \in C^\omega \of{U, K} $ where $ K $ is a compact subset of $Z$.
\end{enumerate}
$U$ is an open and connected neighborhood of zero. The equilibrium zero is assumed to have at least one unstable direction, that is there exists $ 0 \neq p_0 \in Z $ and $ q_n \in U $ converging to zero with $ \Phi \of{n, q_n} = p_0 $. Furthermore assume, that the flow is analytic with respect to the initial data as long as it stays bounded and has values in a compact subset $ K \subset Z $. Then  $ \Phi \of{t, U} $ can not stay uniformly bounded with $ t > 0 $.
\end{lemma}
\begin{proof}
The proof argues by contradiction. Assume that the time $n$--maps stay uniformly bounded, that is consider the (uniformly bounded) family
\[
f_n : U \to K \subset Z, \; \qquad f_n \of{u} := \Phi \of{n, u}.
\]
By Montel compactness theorem there exists a subsequence $ \tilde f_n := f_{\tilde{n} \of{n}} $ of $ f_n $ converging uniformly to an analytic function $ \tilde f $. Then we have for $ \tilde q_n := q_{\tilde{n} \of{n}} $

\[
\norm{p_0} = \norm{\tilde f_n \of{\tilde q_n}} \leq \norm{\tilde f_n \of{\tilde q_n} - \tilde f_m \of{\tilde q_n}} + \norm{\tilde f_m \of{\tilde q_n}}, \qquad m,n \in \N.
\]

For all $ \varepsilon > 0 $, since the convergence is uniform, we can choose, $ m $, $ n $ large enough such that

\[
\norm{\tilde f_n \of{\tilde q_n} - \tilde f_m \of{\tilde q_n}} < \varepsilon.
\]

Now choosing $ n $ large enough, we also have

\[
\norm{\tilde f_m \of{\tilde q_n}}  < \varepsilon.
\]

This implies $ \norm{p_0} < 2 \varepsilon $. Since the argument holds for any $ \varepsilon $ we can conclude $ p_0 = 0 $, which is a contradiction.
\end{proof}

\begin{remark}
The proof even shows that the functions $ \Phi \of{\cdot, U} $ can not stay bounded for any (complex) neighborhood $ \tilde U \subset U $ arbitrarily close to zero for all positive time $ t > 0 $.
\end{remark}

The two Lemmata can be connected the theory of stable and unstable manifolds in dynamicals systems. 

\begin{lemma} \label{lem:unstable_mf_unbounded}
Consider a semi flow $ \Phi : S_Z \subset \R_+ \times Z \to Z $ of the differential equation \eqref{eq:canonical_equation}. 
Assume that the real fast unstable manifold of $ u_+$ is one-dimensional and contains a heteroclinic orbit from $ u_+ $ to equilibrium $ u_- \equiv 0 $. Then solutions on a complex neighborhood in the fast unstable manifold of equilibrium $ u_+ $ can not stay uniformly bounded for real positive time.
\end{lemma}
\begin{proof}
The proof proceeds by contradiction.
\begin{enumerate}[label=(\roman*)]
\item Denote the analytic graph of the one--dimensional real fast unstable manifold by $ \Upsilon $. The real analytic graph can be extended to an analytic function on some complex neighborhood $ U$ of the equilibrium, see Figure \ref{fig:heteroclinic_foliation}.
\item Denote by $ f_n $ the time $n$ - map of the neighborhood $ U $, i.e. 
\[
f_n : U \to Z, \qquad f_n \of{q_0} = \Phi \of{n, (q_0, \Upsilon \of{q_0}}.
\]
where $ (q_0, \Upsilon \of{q_0}) $ is on the heteroclinic orbit. Assume that the family $ f_n $ stays uniformly bounded. \\
Take any sequence $ t_m \in \R $ with $ \abs{t_m} $ small enough and $ t_m \to \delta $. Define $ u_m \in \C $ by
\[
u_m := P_{+} \of{ \Phi \of{t_m, (u_0, \Upsilon \of{u_0})}},
\]
where $ P_{+} $ is the projection on the tangent space of the strong unstable manifold $ W^{su} \of{u_+} $ and  $ \delta $ small enough such that the projection is still contained in $ U$, i.e.
\[
P_{+} \of{ \Phi \of{t_m, (q_0, \Upsilon \of{q_0})}} \in U .
\]
The sequence of $ q_m $ converges to $ q^\delta_0 \in U $. By construction each $ (q_m, \Upsilon \of{q_m}) $ is on the heteroclinic orbit and it holds
\[
\lim_{t \to \infty} \Phi \of{t, q_m} = u_-.
\] 
Since we assumed that the flow stays uniformly bounded, we can apply Lemma \ref{lem:unbounded_solutions_omega_limit} to obtain
\[
\lim_{t \to \infty} \Phi \of{t, U} = u_-.
\]
This contradicts $ u_+ \in U $.
\end{enumerate}
\end{proof}

\begin{remark}
The Lemma shows that, a flow that depends analytically on the initial condition, can not change the limit point if the flow stays uniformly bounded. This implies, that there might be a kinship between grow-up, blow-up and the global attractor of differential equations in the complex domain.
\end{remark}

The Lemma required the mere existence of a real heteroclinic orbit. But we can similarly also prove that there exists grow-up or blow-up even if we do not have such heteroclinic orbit, but a finite-dimensional analytic unstable manifold by Lemma \ref{lem:complex_time_grow_up_semiflow}.

The theorem tells us that the equilibria of analytic systems are related to blow-up or grow-up of dynamical systems in the complexified domain. Next we show that the solution must actually possess blow-up and that the analytic time continuation of the blow-up orbit is also a real time heteroclinic orbit. Indeed, the one-dimensional fast unstable manifold is foliated by heteroclinic orbits and there exists a non-empty boundary to the foliation, which is a blow-up orbit.

\begin{lemma} \label{lem:set_of_heteros}
Consider the setting of Lemma \ref{lem:unstable_mf_unbounded}. Define the set $\tilde H$ as follows
\[
\tilde H := \left \{ u_0 \in U \subset \C, \; \sup_{t \in \R_+} \norm{\Phi \of{t, (u_0,\Upsilon \of{u_0}}}_Z < \infty \right \},
\]
and the set $ H $ as the connected component of $ \tilde H $ that contains $ U \cap \R_- $.
Then $ H $ is foliated by real time heteroclinic orbits with complex initial data and furthermore all real time heteroclinic orbits are time $p$ - path continuations of the real heteroclinic orbit.
\end{lemma}
\begin{proof}
We proceed by the following steps.
\begin{enumerate}[label=(\roman*)]
\item Take any $ q_0 \in U $. Then there exists a $ r > 0 $ such that a small ball $ B_r \of{q_0} $ in the tangent space of the fast unstable manifold is filled by the complex time flow of $ q_0 $.
The proof argues via the complex one-dimensional equation on the fast unstable manifold
\begin{equation} \label{eq:reduced_equation_fast_unstable}
\dot q = \mu q + \tilde f \of{q, \Upsilon \of{q}}, \qquad q \of{0} = q_0.
\end{equation}
where $ \tilde f := P_{su} q  $.
For small enough $ q_0 $ we know that $ \mu q_0 + \tilde f \of{q_0, \Upsilon \of{u_0}} \neq 0 $ for $ q_0 \neq 0 $ and thus the separation of variables formula is well-defined
\begin{equation} \label{eq:reduced_equation_time}
t \of{q, q_0} = \int_{q_0}^{q} \frac{1}{\mu \tau + f \of{\tau, \Upsilon \of{\tau}}} d \tau,
\end{equation}
for any $ q \in  B_r \of{q_0} $. This also implies an explicit estimate on the time $ t \of{q,q_0} $ needed to go from $ q_0 $ to $ q $
\[
\abs{t \of{q, q_0}} \leq 2 \abs{\frac{q-q_0}{\mu q_0 + f \of{q_0, \Upsilon \of{q_0}}}}.
\]
\item The next step is to prove that the set $ H $ is open. Take any $ q_0 \in H $. Then by assumption there exists an $ M > 0 $ such that $\norm{\Phi \of{t, (q_0,\Upsilon \of{q_0}}}_Z < M $. The main problem is that the time $t$ is unbounded. Thus we can not simply take small perturbations of the initial condition and argue by continuous dependence of initial data of time $ t $ maps and let $t$ go to infinity. The idea to remedy that problem is that every initial condition $\tilde q_0 $ in the unstable manifold close to $q_0$ can be reached by solving the reduced equation for a small complex time $q_0$ given by \eqref{eq:reduced_equation_time}. First note, that since the real time flow of $ q_0 $ is uniformly bounded, there exists a $ \delta > 0 $ such that we can extend the solution analytically to complex time strip 

\[
S := \left \{ t \in \C, \abs{\Im \of{t}} \leq \delta \right \}.
\]

Take any $\tilde q_0 \in B_r \of{q_0} $ with $ r > 0 $. By the previous step we can choose an $ r > 0 $ small enough such that the neighborhood $ B_r \of{q_0} $ is obtained by the complex time flow of $ q_0 $ with imaginary part less than $ \delta $. Thus there exists a $ t_{\tilde q_0} := t \of{\tilde q_0, q_0} \in \C $ such that $ u = \Phi \of{t_{\tilde q_0},q_0} $ and $ \abs{\Im \of{t \of{\tilde q_0, q_0}}} \leq \delta $. Due to time analyticity, the complex flow property holds and we obtain
\[
\Phi \of{t,(\tilde q_0, \Upsilon \of{\tilde q_0})} = \Phi \of{t, \Phi \of{t_{\tilde q_0},(q_0, \Upsilon \of{q_0})}} = \Phi \of{t_{\tilde q_0}, \Phi \of{t,(q_0, \Upsilon \of{q_0})}}.
\]
This implies $ \tilde q_0 \in H $.
\item Since $ H $ is open we can apply Lemma \ref{lem:unbounded_solutions_omega_limit} to obtain that $H$ is foliated by real time heteroclinic orbits. Note that the backwards in time convergence is given by definition since $ H $ is a subset of the fast unstable manifold.
\item All in all we can conclude a foliation of the fast unstable manifold as indicated in Figure \ref{fig:heteroclinic_foliation}.
\end{enumerate}
\end{proof}

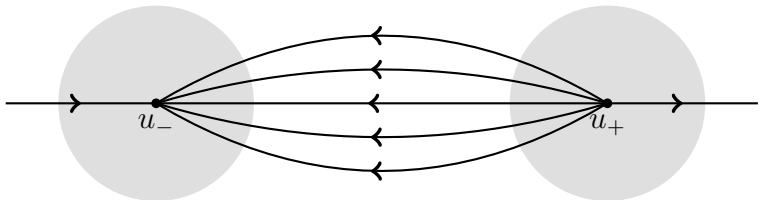
\begin{figure}[h] 
\centering
\begin{tikzpicture}

\draw[fill= lightgray, draw=none, opacity = 0.5] (6,0) circle [radius = 1.3];

\draw[fill= lightgray, draw=none, opacity = 0.5] (0,0) circle [radius = 1.3];

\draw [thick, domain=0:6,decoration={markings, mark=at position 0.5 with {\arrow[ultra thick]{<}}},
        postaction={decorate}] plot ({\x)}, {-0.05*\x*(\x-6)});

\draw [thick, domain=0:6,decoration={markings, mark=at position 0.5 with {\arrow[ultra thick]{<}}},
        postaction={decorate}] plot ({\x)}, {0.05*\x*(\x-6)});

\draw [thick, domain=0:6,decoration={markings, mark=at position 0.5 with {\arrow[ultra thick]{<}}},
        postaction={decorate}] plot ({\x)}, {0.1*\x*(\x-6)});

\draw [thick, domain=0:6,decoration={markings, mark=at position 0.5 with {\arrow[ultra thick]{<}}},
        postaction={decorate}] plot ({\x)}, {-0.1*\x*(\x-6)});

\draw [thick, domain=0:2,decoration={markings, mark=at position 0.5 with {\arrow[ultra thick]{>}}},
        postaction={decorate}] plot ({-2+\x}, {0});

\draw[fill, thick] (0,0) circle [radius=0.05] node[below, thick] {$u_-$};
\draw[fill, thick] (6,0) circle [radius=0.05] node [below, thick ] {$u_+$}; 

\draw [thick, domain=0:6,decoration={markings, mark=at position 0.53 with {\arrow[ultra thick]{>}}},
        postaction={decorate}] plot ({6-\x}, {0});

\draw [thick, domain=0:2,decoration={markings, mark=at position 0.5 with {\arrow[ultra thick]{>}}},
        postaction={decorate}] plot ({6+\x}, {0});

%
%
%
%

\end{tikzpicture}
\caption{Heteroclinic foliation of the fast unstable manifold}  \label{fig:heteroclinic_foliation}
\end{figure}

The set of heteroclinic orbits $ H $ on the fast unstable manifold is open and we study its topological boundary. We show that the boundary is non-empty and that it is a blow-up orbit. The blow-up orbit has a complex time analytic continuation into the positive or negative complex half-plane and it is heteroclinic along time $p$ - paths. 

\begin{figure}[h]
\centering
\begin{subfigure}[b]{.5\textwidth}
\begin{tikzpicture}

\draw[fill= lightgray, draw=none, opacity = 0.5] (6,0) circle [radius = 1.3];

\draw[fill= lightgray, draw=none, opacity = 0.5] (0,0) circle [radius = 1.3];

\draw [thick, domain=0:6,decoration={markings, mark=at position 0.5 with {\arrow[ultra thick]{<}}},
        postaction={decorate}] plot ({\x)}, {-0.05*\x*(\x-6)});

\draw [thick, domain=0:6,decoration={markings, mark=at position 0.5 with {\arrow[ultra thick]{<}}},
        postaction={decorate}] plot ({\x)}, {0.05*\x*(\x-6)});

\draw [thick, domain=0:6,decoration={markings, mark=at position 0.5 with {\arrow[ultra thick]{<}}},
        postaction={decorate}] plot ({\x)}, {0.1*\x*(\x-6)});

\draw [thick, domain=0:6,decoration={markings, mark=at position 0.5 with {\arrow[ultra thick]{<}}},
        postaction={decorate}] plot ({\x)}, {-0.1*\x*(\x-6)});

\draw [color=red,very thick,domain=45:180,decoration={markings, mark=at position 0.5 with {\arrow[ultra thick]{<}}},
       postaction={decorate}] plot ({0.9*cos(\x) + 6}, {0.9*sin(\x)});

\draw [color=blue,very thick, domain=45:180,decoration={markings, mark=at position 0.5 with {\arrow[ultra thick]{<}}},
       postaction={decorate}] plot ({0.9*cos(\x) + 6}, {-0.9*sin(\x)});

\draw [thick, domain=0:2,decoration={markings, mark=at position 0.5 with {\arrow[ultra thick]{>}}},
        postaction={decorate}] plot ({-2+\x}, {0});

\draw[fill, thick] (0,0) circle [radius=0.05] node[below, thick] {$u_-$};
\draw[fill, thick] (6,0) circle [radius=0.05] node [below, thick ] {$u_+$}; 

\draw [thick, domain=0:6,decoration={markings, mark=at position 0.53 with {\arrow[ultra thick]{>}}},
        postaction={decorate}] plot ({6-\x}, {0});

\draw [thick, domain=0:2,decoration={markings, mark=at position 0.5 with {\arrow[ultra thick]{>}}},
        postaction={decorate}] plot ({6+\x}, {0});

%
%
%
%

\end{tikzpicture}
\caption{Path in the unstable manifold}
\end{subfigure}
\begin{subfigure}[b]{.5\textwidth}
\centering
\begin{tikzpicture}

\draw[fill= gray, draw=white, opacity = 0.5] (-1,0) rectangle (5,2);

\draw[fill=white, draw=white, postaction={pattern=my north east lines,pattern color=gray}] (-1,2) rectangle (5,1);
\draw[fill=white, draw=white, postaction={pattern=my north west lines,pattern color=gray}] (-1,0) rectangle (5,1);

%



\draw[line width=0.08cm, draw=red, ->] (0.5,1) -- (0.5,1.8);
\draw[line width=0.08cm, draw=blue,->] (0.5,1) -- (0.5,0.2);

\draw[color=white] (1,0) --(3,0);
\draw[fill] (1,0) circle [radius=0.05cm] node[below] {$T$};
\draw[fill] (3,0) circle [radius=0.05cm] node[below] {$T_1$};

\draw[color=white] (1,2) --(3,2);
\draw[fill] (1,2) circle [radius=0.05cm] node[above] {$T$};
\draw[fill] (3,2) circle [radius=0.05cm] node[above] {$T_1$};

\draw[->] (-0.5,-0.5) -- (-0.5,2.5) node[scale = 0.7,left] {$\Im \of{t} $};

\draw[->] (0,1) -- (5.5,1) node[scale = 0.7,below] {$\Re \of{t} $};

\draw [thick, domain=0:6,decoration={markings, mark=at position 0.55 with {\arrow[ultra thick]{<}}},
        postaction={decorate}] plot ({5-\x}, {1});

\draw [thick, domain=0:6,decoration={markings, mark=at position 0.55 with {\arrow[ultra thick]{<}}},
        postaction={decorate}] plot ({5-\x}, {0.4});

\draw [thick, domain=0:6,decoration={markings, mark=at position 0.55 with {\arrow[ultra thick]{<}}},
        postaction={decorate}] plot ({5-\x}, {0.7});        

\draw [thick, domain=0:6,decoration={markings, mark=at position 0.55 with {\arrow[ultra thick]{<}}},
        postaction={decorate}] plot ({5-\x}, {1});

\draw [thick, domain=0:6,decoration={markings, mark=at position 0.55 with {\arrow[ultra thick]{<}}},
        postaction={decorate}] plot ({5-\x}, {1.6});

\draw [thick, domain=0:6,decoration={markings, mark=at position 0.55 with {\arrow[ultra thick]{<}}},
        postaction={decorate}] plot ({5-\x}, {1.3});

\end{tikzpicture}
\caption{Associated complex time path of real blow-up orbit} 
\end{subfigure}
\caption{Relation between real and imaginary time path in the unstable manifold} \label{fig:idea_of_proof}
\end{figure}
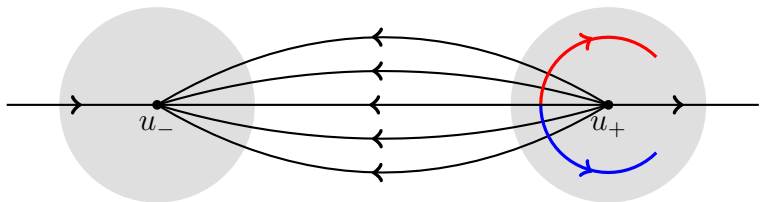
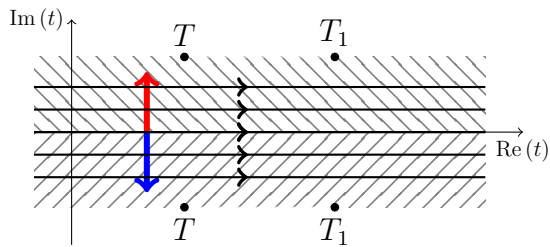

\begin{lemma} \label{lem:boundary_heteroclinic_blow_up}
The set $ \partial H \cap U$ is not empty. The boundary orbits consist of a complex conjugate pair of finite time blow-up orbits with time analytic continuation into the positive or negative complex plane. The time $p$ - paths of the analytic continuation are heteroclinic orbits and contained in $ H$.
\end{lemma}
\begin{proof}
First we prove that $ H$ must have a boundary. As a consequence of Lemma \ref{lem:unstable_mf_unbounded}, we know that there must exist a $ q_0 \in U $ such that the forward flow is unbounded. This element can not be contained in $H$. Assume on the contrary that the boundary of the set $ \partial H \cap U $ is empty. Then $ H $ must be open, connected and closed relative to $ U $ and thus $ H = U $, which can not be true. We can assume, that there exists $ q_0 \in \partial H \cap U $. 
The proof of the Lemma argues by contradiction. Assume that $ q_0 $ was a grow up orbit.
\begin{enumerate}[label=(\roman*)]
\item Take $ \tilde q_0 \in U \cap \R_- $. Arguing as in Lemma \ref{lem:set_of_heteros} we consider the differential equation on the fast unstable manifold
\[
\dot q = \mu q + f \of{q, \Upsilon \of{q}}.
\]
Here we use separation of variables to solve the equation
\[
t_{q_0} = \int_{\tilde q_0}^{q_0} \frac{1}{\mu\tau + f \of{\tau, \Upsilon \of{\tau}}} d \tau.
\]
Thus taking any closed path $ \kappa : \left[0,1 \right] \to U $ with $ \kappa \of{0} = q_0 $ and $ \kappa \of{1} = q_0 $ enclosing the origin we get by Cauchy theorem
\[
t_{q_0}= \oint_{\tilde q_0}^{q_0} \frac{1}{\mu\tau + f \of{\tau, \Upsilon \of{\tau}}} d \tau= \frac{1}{2 i \pi \mu}.
\]
This implies that we can take a purely imaginary time path to encircle zero in the fast unstable manifold. Since multiplication with the imaginary unit corresponds to a rotation of $ \pi /2 $ real and imaginary trajectories are always orthogonal to each other, see Figure \ref{fig:idea_of_proof}.
\item By the above argument we can find a complex time path $ \gamma \of{s} : \left[0, s_0 \right) \to \C $ for some $ s_0 > 0 $ such that 

\[
q \of{\gamma \of{s}, \tilde q_0} \subset H \cap U, \qquad q \of{\gamma \of{s_0}, \tilde{q_0}} = q_0.
\]

By assumption, the real positive time flow for any solution $ \Phi \of{t, (q \of{\gamma \of{s}, \tilde q_0}, \Upsilon \of{q \of{\gamma \of{s}, \tilde q_0}}} $ stays bounded for all $ 0 \leq s < s_0 $ and thus the real and complex time flow commute, i.e.

\[
\Phi \of{ t, (q_0, \Upsilon \of{q_0})} = \Phi \of{ \gamma \of{a} + t, (\tilde q_0, \Upsilon \of{\tilde q_0})}, \qquad t > 0.
\]

For any $ \varepsilon > 0 $ we can choose a $ t_0 > 0 $ large enough such that

\[
\norm{\Phi \of{ t_0 + t, (\tilde q_0, \Upsilon \of{\tilde q_0})}} < \varepsilon, \qquad t > 0.
\]

This implies for $ \varepsilon $ small enough

\[
\norm{\Phi \of{ t_0 + t, (q_0, \Upsilon \of{q_0})}} = \norm{\Phi \of{ \gamma \of{a},\Phi \of{t_0 + t, (\tilde q_0, \Upsilon \of{\tilde q_0})}}} < 2 \varepsilon,
\]

for all $ t > 0 $. This contradicts the grow-up assumption.
\item Note that this also proves, that the one sided imaginary time flow of $ q_0 $ is contained in $H$.
\end{enumerate}
\end{proof}

This gives an estimate for the maximal complex time strip where the function can stay regular.

\begin{corollary}
The real heteroclinic orbit has blow-up in a strip of size $ \frac{1}{4 \pi \mu} $.
\end{corollary}

\section{The non-linear heat equation} \label{sec:nonlin_heat}

In this section, we will focus on the complex nonlinear heat equation with quadratic nonlinearity
\begin{equation} \label{eq:nonlinear_heat_quad}
u_t = u_{xx} + u^2, \qquad x \in \left(-1, 1 \right), \; u \of{\pm 1} = 0.
\end{equation}

Consider the space $ Z := H^2 \of{I} \cap H_0^1 \of{I} $, $ Y := H^1_0 \of{I} $ and $ X := L^2 \of{I} $. Here all the functions are complex valued. The Laplace operator $ A := \partial_{xx} $ is a bounded operator from $ Z $ to $ X$ and generates an analytic semigroup  \citep{lunardi95},
\[
A : D \of{A} := H^2 \of{I} \cap H_0^1 \of{I} \subset L^2 \of{I} \mapsto L^2 \of{I}.
\]
By Sobolev embedding theorems, the quadratic nonlinearity $ f \of{u} := u^2 $ is analytic from $ f : Z \to Y $ as well as a function $ f : Y \to X $. By \citep{lunardi95}, \citep{henry93} the local solution is analytic in time with values in $ Z $. 

We give a more advanced description of the time analytic continuation of the blow-up orbit of Lemma \ref{lem:boundary_heteroclinic_blow_up}. The nonlinear heat equation has an unique positive equilibrium $ u_+ $ and it is well-known that solutions starting above $ u_+ $ exhibit finite time blow-up. Furthermore, the fast unstable manifold is one-dimensional and solutions that start below $ u_+ $ converge to zero.

We follow the idea of \citep{guo12} to disprove blow-up of solutions on the unstable manifold of $u_+$ with complex initial conditions. In particular, we show, that the real blow-up orbit stays uniformly bounded on time $p$-paths with non-vanishing imaginary part. Note, that along time $p$-paths the equation \eqref{eq:nonlinear_heat_quad} can be written as real system of real and imaginary part for $ u = v + i w $

\begin{align} \label{eq:system_nonlinear_heat}
\begin{split}
v_t &= v_{xx} + v^2 - w^2, \\
w_t &= v_{xx} + 2 v w.
\end{split}
\end{align}

Different from \citep{guo12}, we consider Dirichlet boundary conditions, but we use the same maximum principle to show boundedness of solutions. The system \eqref{eq:system_nonlinear_heat} possesses a maximum principle \citep{weinberger75}, \citep{redheffer80}, \citep{evans10}. One corollary from these maximum principles is that the solutions of systems of the form $ w_t = \Delta w + f \of{w} $ where $ w \in \R^2 $ is contained in the ODE solution $ \dot w = f \of{w} $ if the ODE solution defines a convex set and the initial condition and boundary condition lies in the interior of the convex set. \\
The ODE solutions of equation \eqref{eq:system_nonlinear_heat} are circles, see Figure \ref{fig:invariant_ode_solution}. We prove that the non-real initial condition on the fast unstable manifold of $ u_+ $ is contained in such circles.

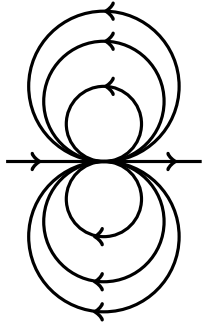
\begin{figure}[h]
\centering
\begin{tikzpicture}
     
\draw[very thick, decoration={markings, mark=at position 0.25 with {\arrow[ultra thick]{>}}},
        postaction={decorate}] (0,1) circle [radius=1];
\draw[very thick, decoration={markings, mark=at position 0.75 with {\arrow[ultra thick]{<}}},
        postaction={decorate}] (0,-1) circle [radius=1];
\draw[very thick,decoration={markings, mark=at position 0.25 with {\arrow[ultra thick]{>}}},
        postaction={decorate}] (0,0.5) circle [radius=0.5];
\draw[very thick, decoration={markings, mark=at position 0.75 with {\arrow[ultra thick]{<}}},
        postaction={decorate}] (0,-0.5) circle [radius=0.5];
\draw[very thick, decoration={markings, mark=at position 0.25 with {\arrow[ultra thick]{>}}},
        postaction={decorate}] (0,0.8) circle [radius=0.8];
\draw[very thick, decoration={markings, mark=at position 0.75 with {\arrow[ultra thick]{<}}},
        postaction={decorate}] (0,-0.8) circle [radius=0.8];

\draw[very thick, decoration={markings, mark=at position 0.75 with {\arrow[ultra thick]{>}}},
        postaction={decorate}] (0,0) -- (1.3,0);
\draw[very thick, decoration={markings, mark=at position 0.75 with {\arrow[ultra thick]{<}}},
        postaction={decorate}] (0,0) -- (-1.3,0);

\end{tikzpicture}
\caption{Solution to the ODE flow} \label{fig:invariant_ode_solution}
\end{figure}

\subsection*{ODE results}

The maximum principle of the parabolic system \eqref{eq:quadratic_system} allows to use the ODE flow to the quadratic equation $ \dot U = U^2$ to show a priori estimates on PDE solutions. 

\paragraph*{Explicit solution of the ODE}

Consider the differential equation

\[
\dot U \of{t} = U^2 \of{t}, \qquad t \in \C,\; U \of{0} = U_0 \in \C,
\]

The equation has an explicit solution, which can be obtained by separation of variables. 

\begin{equation} \label{eq:quadratic_ode_explicit}
U \of{t} := \frac{1}{1/U_0 - t}.
\end{equation}

The flow $ \eta \of{t, U_0} $ is defined as the solution $ U \of{t} $. The solution maps straight lines in $ t = r e^{i \theta} $, $ r \in \R $ to circles, if it exists as a map to the complex plane $ \C$.

But one can also consider the solution $ U \of{t} $ as an analytic function from the complex plane to the Riemann sphere $ \hat \C $ for all $ t \in \C $. The Riemann sphere is the set $ \C \cup \left \{ \infty \right \} $ together with the charts $ \tilde U_1 := U $ and $ \tilde U_2 = U^{-1} $. The two charts coincide almost everywhere. The domain of definition of each chart is $ \C $. Using the second chart, i.e. introducing $ \tilde{U}_2 \of{t} := U_1^{-1} \of{t} $ yields the differential equation
\[
\dot{ \tilde U}_2 \of{t} = -1, \qquad t \in \C, \qquad (\tilde U_2)_0 = U_0^{-1}.
\]
This has the trivial solution $ \tilde U_2 \of{t} = U_0^{-1} - t $, which is the same as \eqref{eq:quadratic_ode_explicit}. Note that the solution even exists if $ \tilde U_2 = 0 $, that is $ \abs{U} = \infty $. This implies, that the imaginary part of $ \Im \of{\tilde U_2 \of{t}} = \Im \of{U_0^{-1}} $ is a conserved quantity for the real time flow. To study the behaviour for unbounded $ t $ one has to use the chart $ \tilde U_1 $, in which all orbits converge to zero. The point zero is the only point that can not be described in the chart $ \tilde U_2 $. Since there is a conserved quantity in the $ \tilde{U_2} $ chart, there also exists a conserved quantity in the $ \tilde{U}_1 $ chart. 

Writing the function $ U \of{t} := V \of{t} + i W \of{t} $, it can also be seen by explicit calculation the quantity
\[
H \of{t} := V^2 \of{t} + \of{W \of{t} - W_0^{-1}}^2,
\]
 is conserved by the flow. This gives the following definition.

\begin{definition}[Solution circle/disk]
A circle of the form

\[
C_{v_0} := \left \{ (u, v) \in \R^2, \; u^2  + (v-v_0^{-1})^2 = R^2 \right \} ,
\]
is a solution to the ODE and called solution circle. The interior of the circle is called solution disk.
\end{definition}

\paragraph*{Some ODE results}

We use properties of the flow $ \eta $ to prove invariance principles of the PDE. The proofs are based on an idea of \citep{guo12}.

\begin{lemma}
Consider a straight half line $ \gamma \of{s} := s e^{i \phi} $, $ s \geq 0, \; 0 < \phi < \pi $ of initial conditions.
Then the curve $ z \of{s} := \eta \of{t, \gamma \of{s}} $, where $ \eta $ is the ODE flow of $ \dot U = U^2 $, is convex for each fixed $ t > 0 $.
\end{lemma}
\begin{proof}
Recall that a plane curve $ s \mapsto z \of{s} \in \C $ is convex if the curvature

\[
\kappa \of{s} := \Re \of{ i z' \bar z ''},
\]

does not change sign. Explicit calculation of $ \kappa \of{s} $ yields
\[
\kappa \of{s} = \frac{2 t \sin(\phi)}{(\of{s t - \cos \phi}^2+ \sin(\phi)^2)^3}.
\]
Since the straight line starts in the upper half plane, i.e. $ 0 < \phi < \pi $, we have $ \sin \phi > 0 $ and the curvature is always positive for $ t > 0 $, see Figure \ref{fig:invariant_set}.
\end{proof}

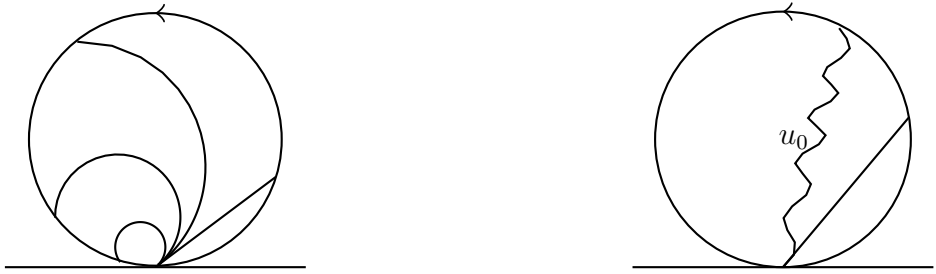
\begin{figure}[h]
\begin{subfigure}[t]{.5\textwidth}  
\centering
\scalebox{2}{\begin{tikzpicture}
     
\draw[decoration={markings, mark=at position 0.25 with {\arrow{>}}},
        postaction={decorate}] (0,0.85) circle [radius=0.84];

\draw (-1,0) -- (1,0);
\draw (0,0) -- (0.8,0.6);


\draw [domain=40:0.99, samples=400] plot ({(0.8*\x-1)/((0.8*\x-1)*(0.8*\x-1)+0.6*\x*0.6*\x)}, {0.6*\x/((0.8*\x-1)*(0.8*\x-1)+0.6*\x*0.6*\x)});

\draw [domain=40:0.99, samples=400] plot ({(0.8*\x-2)/((0.8*\x-2)*(0.8*\x-2)+0.6*\x*0.6*\x)}, {0.6*\x/((0.8*\x-2)*(0.8*\x-2)+0.6*\x*0.6*\x)});

\draw [domain=40:1, samples=400] plot ({(0.8*\x-5)/((0.8*\x-5)*(0.8*\x-5)+0.6*\x*0.6*\x)}, {0.6*\x/((0.8*\x-5)*(0.8*\x-5)+0.6*\x*0.6*\x)});

\end{tikzpicture}}
\caption{Evolution of straight line in invariant region for $ 0 = t_0 < t_1 < t_2 < t_3$} \label{fig:invariant_set}
\end{subfigure}
\begin{subfigure}[t]{.5\textwidth}  
\centering
\scalebox{2}{\begin{tikzpicture}
     
\draw[decoration={markings, mark=at position 0.25 with {\arrow{>}}},
        postaction={decorate}] (0,0.85) circle [radius=0.85];
       postaction={decorate}] (0,-0.8) circle [radius=0.8];

\draw (-1,0) -- (1,0);
\draw (0,0) -- (0.84,1);

\draw [domain=0:0.83] plot ({0.5*\x+0.07*sin(2000*\x)
+0.01*sin(8002*\x)}, {2*\x+0.07*sin(330*\x)});
\draw (-0.1,0.85) circle [radius=0cm] node[right, scale=0.5] {$u_0 $};

\end{tikzpicture}}
\caption{Initial condition in invariant region}  \label{fig:invariant_set_initial_condition}
\end{subfigure}
\caption{Invariant region in complex domain}
\end{figure}

The straight line is transported by the ODE flow and traces a convex domain, see Figure \ref{fig:invariant_set_initial_condition}. A solution, which is contained in a solution disk and starting to the left of a straight line, must converge to zero. This idea is due to \citep{guo12} and allows to prove the following a priori estimate on the analytic continuation.

\begin{lemma} \label{lem:straight_line_sub_sol}
Assume that the initial condition $ u_0 \of{x} := v_0 \of{x} + i w_0 \of{x} $ is contained in the solution disk of $ s_0 e^{i \varphi} $, $ 0 < \varphi < \pi $ and $ 0 < \max \arg u_0 \of{x} < \pi $, see Figure \ref{fig:invariant_set_initial_condition}. Then the ODE solution $ \eta \of{t, u_0 \of{x}}$ satisfies the following inequality 
\[
\norm{\eta \of{t, u_0 \of{x}}}_\infty \leq \begin{cases}  \of{\of{\alpha/s_0 -t}^2 + \beta^2/s_0^2}^{-1/2} & \textrm{ for } \alpha t < 1/s_0, \\ \of{\alpha \beta t}^{-1} & \textrm{ for } \alpha t \geq 1/s_0. \end{cases}
\]
with $ \alpha := \cos \phi $ and $ \beta := \sin \phi $.
\end{lemma}
\begin{proof}
By the explicit representation of $ \eta \of{t, s e^{i \varphi}} $, we obtain

\[
\abs{\eta \of{t, s e^{i \varphi}}} = \frac{1}{\of{r \alpha - t}^2 + r^2 \beta^2}.
\]

where $ r = s^{-1} $, $ s_0^{-1} \leq r < \infty $.

The derivative with respect to $r$ is

\[
\partial_r \eta \of{t, s e^{i \varphi}} = \frac{2 \of{r \alpha-t} \alpha + 2 r \beta^2}{\of{\of{r \alpha - t}^2 + (r \beta)^2}^2} r^2.
\]

Setting the derivative to zero we obtain

\[
2 \of{r \alpha-t} \alpha + 2 r \beta^2 = 0 \Rightarrow r = t\alpha.
\]

Thus as long as $ \alpha t < 1/s_0 $ we have $ r = r_0 $, i.e.
\[
\norm{\eta \of{t, u_0 \of{x}}}_\infty \leq \of{\of{r_0 \alpha -t}^2 + r_0^2 \beta^2}^{-1/2},
\]
and for $ \alpha t > r_0 $ we have
\[
\norm{\eta \of{t, u_0 \of{x}}}_\infty \leq \frac{1}{\alpha \beta t}.
\]
In particular does the solution converge to zero as expected.
\end{proof}

\begin{lemma} \label{lem:solution_left_line}
Assume, that the smooth function $ u_0 \of{x} := v_0 \of{x} + i w_0 \of{x} : \left(-1, 1 \right) \to \C $, $ u_0 \of{\pm 1} = 0 $ has positive imaginary part and positive boundary derivative. Then an angle $ 0 < \varphi < \pi $ exists, such that $ u_0 \of{x} $ is to the left of the half-line $ z \of{s} := s e^{i \varphi} $. The angle $ \varphi $ satisfies 
\[
\varphi = \min_{ x \in I} \arctan \of{\frac{w_0 \of{x}}{v_0 \of{x}}}.
\]
\end{lemma}
\begin{proof}
On any compact subset $ K \subset I $, $ w_0 \of{x} $ is uniformly bounded from below and there exists a $ \delta > 0 $ such that
\[
\delta < \arctan \of{\frac{w_0 \of{x}}{v_0 \of{x}}} < \pi - \delta, \qquad x \in K.
\]
The only problem may occur at the boundary. Expansion of $ x $ close to the boundary gives for $  x = \pm 1 + y $.

\begin{equation} \label{iq:boundary_r}
\frac{\mp (w_0)_x \of{\pm 1} y + C y^2}{\mp (v_0)_x \of{\pm 1} y + C y^2} > c_1 > 0.
\end{equation}
for $ y $ small enough. This implies, that $ \varphi $ is positive and less than $ \pi$.
\end{proof}

\begin{lemma} \label{lem:trapped_solution_positive_im}
Let the smooth function $ u_0 \of{x} := v_0 \of{x} + i w_0 \of{x} : \left(-1, 1 \right) \to \C $, $ u_0 \of{\pm 1} = 0 $ have positive imaginary part and positive boundary derivative. Then the solution is contained in an invariant ODE disk.
\end{lemma}
\begin{proof}
The proof is similar to the proof above. Geomtrically it is already clear, that since the initial condition is contained in a cone in the upper half-plane, we can find in invariant ODE disk.
\begin{enumerate}[label=(\roman*)]
\item We need that the curve $ \gamma \of{x} := (v_0 \of{x}, w_0 \of{x}) $  is contained in a circle around $ i R $ of radius $ R $ for some $ R > 0 $. Thus we have to calculate
\[
R^* := \sup_{x \in I} R \of{v_0 \of{x}, w_0 \of{x}},
\]

where $ R \of{y,z} $ is defined as follows
\[
R \of{y,z} := \frac{y^2+ z^2}{2 z}.
\]

The expression for $ R \of{y,z} $ comes from the equation $ y^2 + (R-z)^2 = R^2 $ solved for $ R $ and describes the circle around $ i R $  containing the point $(y,z) $ and the origin.
We need to show that $ R^* $ is bounded.
\item Consider any compact subset $ K \subset I $. Define $ \alpha := \min_{K} \varphi \of{x} > 0$. This implies
\begin{equation} \label{iq:compact_r}
\sup_{x \in I} R \of{v_0 \of{x}, w_0 \of{x}} = \sup_{x \in I} \frac{v_0 \of{x}^2 + w_0 \of{x}^2}{2 v_0 \of{x}} \leq \frac{4 \norm{u_0}^2_\infty}{ \min_{x \in K} w_0 \of{x}} < \infty .
\end{equation}
\item The only problem can occur at the boundary of $ I $. Here we obtain again by expansion close to the boundary with $ x = \pm 1 + y $, $ y > 0 $ small enough
\begin{equation} \label{iq:boundary_r}
R \of{v_0 \of{x}, w_0 \of{x}} \leq \frac{(v_0)^2_x \of{\pm 1} y^2 + (w_0)^2_x \of{\pm 1} y^2 + C y^3}{ \mp (w_0)_x \of{ \pm 1 } y + C y^2} < c_1 y.
\end{equation}
\end{enumerate}
\end{proof}

\subsection*{Analysis of the PDE}

Denote by $ u_+ $ the unique positive, symmetric equilibrium of stationary problem

\[
u_{xx} \of{x} + u^2 \of{x} = 0, \qquad u \of{\pm 1} = 0.
\]

The linearisation is hyperbolic and has simple eigenvalues and a one-dimensional analytic fast unstable manifold $ W^u $. The graph of the unstable manifold $ \Upsilon \of{\tau} $ is analytic, see \citep{henry93}.

\begin{lemma} \label{lem:orbit_bounded_hetero}

There exists a $ r > 0 $ such that the positive time flow $ \Phi \of{t, (\tau, \Upsilon \of{\tau}} $, $ t > 0 $ is bounded for $ \tau \in \C \setminus \R_+$ and $ \abs{\tau} < r$.
\end{lemma}
\begin{proof}
The proof is based on the ODE results.
\begin{enumerate}[label=(\roman*)]
\item The $L^2$-normalized eigenfunction $ \varphi \of{x} $ to the fast unstable manifold is positive \citep{quittner07}. It satisfies the equation
\[
\varphi_{xx} + 2 u_+ \varphi = \mu \varphi, \; \varphi \of{ \pm 1} = 0, \; \mu > 0.
\]
for some $ \mu > 0 $. Thus by Hopf lemma we obtain 
\[
\min \left \{ \mp \varphi_x \of{\pm 1} \right \} = \delta > 0.
\]
\item Assume first that $ \tau < 0 $. Then $ u_+ \of{x} + \tau \varphi \of{x} + \Upsilon \of{\tau} \leq u_+ \of{x} $ for all $ x \in I $ and $ \abs{\tau} $ small enough. This implies, that the solution converges to zero for $ \tau < 0 $.
\item Consider the initial data $ u_0 \of{\tau} := u_+ + \tau \varphi + \Upsilon \of{\tau} $ with $ \tau \in \C \setminus \R $ and $ \abs{\tau} < r $. The graph $ \Upsilon $ is of quadratic order at zero. This implies, that for small enough $ \tau $ the imaginary part of $ u_+ + \tau \varphi + \Upsilon \of{\tau} $ is positive and has positive boundary derivative.
\item The claim follows by Lemma \ref{lem:trapped_solution_positive_im}.
\end{enumerate}
\end{proof}

\begin{corollary}
The real time flow on the complex fast unstable manifold stays bounded in the slit disk of the fast unstable manifold.
\end{corollary}


\begin{remark}
From the perspective of Lemma \ref{lem:boundary_heteroclinic_blow_up} we have shown, that the boundaries of the two complex conjugated heteroclinic nests (see Figure \ref{fig:heteroclinic_foliation}) coincide. More precisely it is the positive real blow-up orbit. Furthermore note, that we also have proven, that the every heteroclinic obtained in Lemma \ref{lem:orbit_bounded_hetero} is actually a time $p$ -- path of the single real time heteroclinic orbit.
\end{remark}

\paragraph*{Continuation back to the real axis}

For the nonlinear heat equation one can extend the heteroclinic orbit up to the real axis, but not back onto the real axis as indicated in Figure \ref{fig:heteroclinic_foliation} directly after blow-up. \\
Naturally the question arises what happens to the analytic continuation of the real heteroclinic orbit along time $p$ -- paths. We can show, that even if we are not able to continue the solution back to the real axis immediately after the blow-up, we can continue back to the real axis after some time $ T_1 > 0 $.

In this paragraph we prove that the solution can be continued to the regions $ S_\pm $, see Figure \ref{fig:complex_time_path}.

\begin{figure}[h] 
\centering
\begin{tikzpicture}



\draw[fill=white, draw=white, postaction={pattern=my north east lines,pattern color=gray}] (-1,1) -- (-1,2) -- (3,2)  -- (3,0) -- (3,1) -- (-1,1);

\draw[fill=white, draw=white, postaction={pattern=my north west lines,pattern color=gray}] (-1,1) -- (-1,0) -- (3,0)  -- (3,1) -- (-1,1);

\draw[draw=white,GridSize=10pt, postaction={pattern=MyGrid,pattern color=gray}] (3,2) -- (5,2) -- (5,0) -- (3,0) -- (3,2) ;
%



\draw (0,1.5) node[fill=white] {$S_+$};
\draw (0,0.5) node[fill=white] {$S_-$};

\draw[->] (-0.5,-0.5) -- (-0.5,2.5) node[scale = 0.7,left] {$\Im \of{t} $};

\draw (1,1) circle [radius=0.07cm] node[fill=white, below] {$T$};
\draw (3,1) circle [radius=0.07cm] node[fill=white, below] {$T_1$};

\draw[->] (-1.5,1) -- (5.5,1) node[scale = 0.7,below] {$\Re \of{t} $};

\draw[line width = 0.08cm, color=white] (1,1) --(3,1);

\draw (1,1) circle [radius=0.07cm];
\draw (3,1) circle [radius=0.07cm];

\draw (-1,2) -- (5,2) node[right, scale = 0.7] {$\delta$};
\draw (-1,0) -- (5,0) node[right, scale = 0.7] {$-\delta$};

\end{tikzpicture}
\caption{Domain of existence of continued solutions} \label{fig:complex_time_path}
\end{figure}
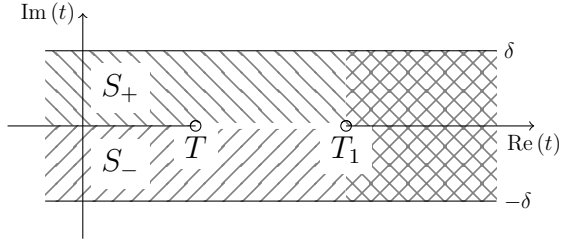

This is already true, due to Lemma \ref{lem:boundary_heteroclinic_blow_up} since it implies, that solutions along time $p$-pathes are heteroclinic orbits. Once the solution is close to zero, which is completely stable, we can tilt the time path to obtain a solution on the real time axis. But in this paragraph, we derive an upper estimate for the time resurrection $ T_1 $.

The next to Lemma give a lower bound on the blow-up rate. The proofs are very similar to the real-valued case in \citep{quittner07}.

\begin{lemma} \label{lem:blow_up_time_initial_data}
If $ \nu_0 := \norm{u_0}_\infty < r $. Then the flow $ \Phi \of{t, u_0} $ is regular along $ t = r e^{i \theta} $ for $ \abs{\theta} < \frac{\pi}{2} $ for at least time $ r = \frac{1}{C^2 \nu_0} $, where $ C > 0 $ is independent of $ \nu_0 $.
\end{lemma}
\begin{proof}
Since the Laplace operator generates an analytic semigroup in the space of continuous functions \citep{lunardi95}, we solve the variation of constants formula
\[
u \of{t} = T \of{t} u_0 + \int_0^z T \of{z-s} u \of{s}^2 ds.
\]
Taking the sup-norms we obtain
\[
\nu \of{t} := \norm{u \of{t}}_\infty \leq C \nu_0 + C \int_0^r  \nu^2 \of{s} ds.
\]

since $ t $ is contained in the sector of the left half-plane. We can compare this inequality to the solution $ \eta \of{t, C^2 \nu_0} / C $. The difference of the solutions satisfies the inequality

\[
h \of{r} := \eta \of{r, C^2 \nu_0}/C - q \of{r} \geq  C \int_0^r  \of{q \of{s} + \eta \of{s,C^2 q_0} / C } h \of{s}  ds,
\]

Since $ h \of{0} > 0 $ we have $ h \of{s} > 0 $ for all $ 0 < s < r $. By Gronwall inequality $ h \of{s} < \infty $ as long as $ q \of{s} + \eta \of{s,C^2 \nu_0}/C < 2 \eta \of{s,C^2 \nu_0}/C < \infty $. But $ \eta \of{s,C^2 \nu_0} / C $ exists up to time $ r = \frac{1}{C^2 \nu_0} $.
\end{proof}


The lower blow-up rate estimate gives an estimate on the blow-up time.

\begin{corollary} \label{cor:uniform_bounded_small}
Suppose that $ \abs{t-T}^p \norm{u \of{t - T}}_\infty \leq M < \infty $ for some $ 0 < p < 1 $. Then does the solution $ u \of{t,x} $ stay regular until $ T $.
\end{corollary}
\begin{proof}
Argue by contradiction. By Lemma \ref{lem:blow_up_time_initial_data} and assumptions, it holds,
\[
\abs{T-t} \geq \frac{1}{C^2 \norm{u \of{T-t}}_\infty} \geq  \frac{\abs{T-t}^{p}}{M C^2 }.
\]
Since $ p < 1 $, this gives a contradiction for $ t $ close enough to $ T$.
\end{proof}

The two Lemmata above give a lower bound on the existence time depending on the initial condition. 
This allows to prove, that the analytic continuations of the blow-up orbit can be continued back to the real axis once they are small enough. We show first that this is not true without the Laplace operator.

\begin{proposition}
Consider the ODE flow of
\begin{align} \label{eq:nonlinear_heat_profile_ode}
\begin{split}
u_t \of{t,x} &= u^2 \of{t,x}, \qquad x \in I := \left(-1, 1 \right) \\
u \of{0,x} &:= u_0 \of{x} \in C^0 \of{I,\R_+}, \; u_0 \not \equiv 0, 
\end{split}
\end{align}
and Dirichlet boundary conditions.
Then the solution to equation \eqref{eq:nonlinear_heat_profile_ode} exists for all $ t \in \C \setminus \left[T, \infty \right)$ and is unbounded on $ \left[T, \infty \right)$, where $ T := \frac{1}{\max_{x \in I} u_0 \of{x}} $.
\end{proposition}
\begin{proof}
We can solve equation \eqref{eq:nonlinear_heat_profile_ode} explicitly and calculate the blow-up time for each point $x \in I $,
\[
u \of{t,x} = \eta \of{t,u_0 \of{x}} = \frac{1}{\frac{1}{u_0 \of{x}}-t}.
\]

Thus the blow-up time $ T \of{x} $ is $T \of{x} = \frac{1}{u_0 \of{x}}$.

Note, that $ x \to \partial I $ implies $ T \of{x} \to \infty $ and also $ T = \min_{x \in I} t_b \of{x} = \frac{1}{\max_{x \in I} u_0 \of{x}} $. Now the claim follows by the intermediate value theorem.
\end{proof}

Let us compare the ODE flow \eqref{eq:nonlinear_heat_profile_ode} with the PDE semiflow $ \phi \of{t, u_0} $, for real initial conditions $ 0 \leq u_0 \neq 0 $, but allowing for, both, real and complex times $ t $. We keep Dirichlet conditions in either case. We have seen how the ODE solution $ u \of{t,x} $ must blow up whenever $t$ returns to the real axis after $ T = \frac{1}{\max_{x \in I} u_0 \of{x}} $, no matter which complex detour our time path $t$ might take. Due to the presence of the Laplacian $ u_{xx} $, however, the PDE solution $ u \of{t,x} $ may behave quite different. We have already seen how the Laplacian generates complete blow-up in the real domain, due to infinite propagation speed outwards from the singularity. In the complex time domain, in contrast, the Laplacian prevents quadratic tangencies at the boundary, via the Hopf Lemma. Moreover the initially real spatial profile $ u_0 \of{x} $ is pushed into a sector in the complex upper half plane, for small imaginary times $ t = i \delta $, provided that $ u_0 $ satisfies

\[
(u_0)_{xx} + f \of{u_0} \geq 0.
\]
Immediately afterwards, comparison with a straight line solution $ \eta \of{t, s \exp \of{i \phi}} $ strikes, and traps the solution for all later real times,

\[
t = i \delta + \tau , \qquad \tau \geq 0. 
\]

We formulate this simple argument in the following theorem. The uniform boundedness, and decay, established here will allow us to even return to the real axis itself, boundedly.

Also note, the proof is completely different from Masuda \citep{masuda84}. He used the explicit spatial constant solution to show boundedness on the real axis again. For initial conditions $ u_0 $ that are close to a spatial constant profile, the ODE flow $ \eta \of{t, u_0 \of{x}} $ is also uniformly bounded for $ t \in \C \setminus \left(\frac{1}{\max_{x \in I} u_0 \of{x}}, \frac{1}{\min_{x \in I} u_0 \of{x}} \right) $. In particular, for spatially constant solutions $ \frac{1}{\max_{x \in I} u_0 \of{x}} = \frac{1}{\min_{x \in I} u_0 \of{x}} $ and the solution just blows up at a single time point.  Masuda \citep{masuda84} showed that solutions, which are sufficiently close to a spatially constant solution the Laplacian can not desynchronize the blow-up much for different $ x$. But he has also shown, that time analyticity is immediately destroyed by the Laplacian operator.

\begin{theorem} \label{thm:continuation_back_to_the_real_axis}
Let $ 0 \leq u_0 \of{x} = u_+ \of{x} + w \of{x} $ with $ w_{xx}  \of{x} + w^2 \of{x} \geq 0 $ and $ w \not \equiv 0 $. Then the flow blows up at some finite time $ 0 < T < \infty $ completely. However, there exists a upper and lower spall strip $ S \of{\delta, \left[T, T_1 \right]} $ with $ T_1 < 2 C_0 \max \left\{ \max_{x \in I} \frac{v_0 \of{x}}{w_0 \of{x}}, \norm{u_+}_\infty \right \} $ to which the solution can be extended. The constant $ C_0 > 0 $ depends only on the heat semigroup. In particular, it can be continued after time $ T_1 $ back to the real axis.
\end{theorem}
\begin{proof}
By the maximum principle we can assume that there exists a time $ 0 < t_0 < T $ such that $ u_t \of{t_0, x} \geq 0 $. By the Hopf Lemma $u_t \of{t_0, x}$ has a positive boundary derivative, i.e.

\[
\pm u_{tx} \of{t_0, \pm 1} > c_0 > 0.
\]

We can expand the function around $ t = t_0 + \tau $ with respect to $ \tau \in \R$, $ \abs{\tau} < \varepsilon $,

\[
u \of{t_0 + i \tau, x} = u \of{t_0, x} + i \tau u_t \of{t_0,x} - \tau^2 g \of{\tau, x} .
\]

where $ g \of{\tau, x} $ is bounded and differentiable with respect to $ x$. The imaginary part of the boundary derivative is

\[
\pm \Im \of{u_x \of{t_0 + \tau, \pm 1}} =  \pm \tau u_{tx} \of{t_0, \pm 1} \mp \tau^2 \Im \of{g_x \of{\tau, \pm 1}} > c_0 \tau + M \tau^2.
\]

Choosing $ \varepsilon > 0 $ small enough, implies that the imaginary part of the boundary derivative is positive. Thus, we can apply Lemma \ref{lem:trapped_solution_positive_im} to show that the solution exists for any positive $ \tilde t \in \R_+ $, i.e. 
\[
\norm{u \of{t_0 + i \tau + \tilde t, \cdot} }_\infty \leq M < \infty .
\]

From Lemma \ref{lem:solution_left_line} we know, that the solution function $ u \of{t_0 + i \tau} $ is to the left of a straight line of angle $ \varphi > C \tau $ for $ C =  \min_{x \in I} \frac{w_0 \of{x}}{2 v_0 \of{x}} $. Furthermore it is contained in a solution disk to $ s_0 e^{i \varphi} $ with $ s_0 < \norm{u \of{t_0 + i \tau, \cdot} }_\infty < 2 \norm{u_0}_\infty $ for $ \tau $ and $ t_0 $ small enough.

From Lemma \ref{lem:straight_line_sub_sol} the solution satisfies the following bound

\[
\norm{u \of{t_0 + i \tau + \tilde t, \cdot} }_\infty \leq \frac{1}{\tilde t \cos \varphi \sin \varphi } \leq \frac{1+\varepsilon}{C \tilde t \tau}.
\]

for $ \tilde t > 2 \norm{u_0}_\infty $ and any $ \varepsilon > 0 $, if one chooses $ \tau $ small enough.

Lemma \ref{lem:straight_line_sub_sol} guarantees the existence of solutions, when solving along a slanted time line $ t = (1-i) r / \sqrt{2} $, if the sup-norm of the initial condition is less than, i.e. if $ \norm{u_0}_\infty < \frac{1}{\sqrt{2} C_0 r} $. Thus, we can solve back to the real axis, boundedly, if  

\[
 \frac{\sqrt{2} \tau}{1+\varepsilon} <  C \tilde t \tau \Rightarrow  \frac{\sqrt{2}}{1+\varepsilon} < C \tilde t.
\]

This implies that we need to wait for time 

\[
T_1 = 2 C_0 \max \left\{ \max_{x \in I} \frac{v_0 \of{x}}{w_0 \of{x}}, \norm{u_+}_\infty \right \} < \infty,
\] 

until we can return back to the real axis. The solution exists on an upper and lower spall strip

\[
S_\pm \of{ \delta, \left[T, T_1 \right]}.
\]

\end{proof}

\begin{remark}
We can choose in particular $ w = \varphi $, where $ \varphi $ is the first eigenfunction at $ u_+ $.
\end{remark}

This time estimate is unfortunately far from optimal. One might expect continuation to spall strip where $ T = T_1$, i.e. the real blow up happens at an isolated time point in the complex plane. At the present we are not able to resolve the question.

\paragraph*{Description of continued solutions}

Using the Cauchy formula we can prove that blow-up solutions on the fast unstable manifold can not coincide after blow up again. 

\begin{figure}[h] 
\begin{subfigure}[t]{.5\textwidth}  
\centering
\begin{tikzpicture}


%
%

\draw (0,1.5) node {$\Gamma_+$};
\draw (0,0.5) node {$\Gamma_-$};

\draw[very thick, ->,color=blue] (0.5,1) -- (0.5,0.4)--(3.5,0.4) -- (3.5,1);
\draw[->,very thick,color=red] (0.5,1) -- (0.5,1.6)--(3.5,1.6) -- (3.5,1);

\draw[->] (-0.5,-0.5) -- (-0.5,2.5) node[scale = 0.7,left] {$\Im \of{t} $};

\draw[->] (-1.5,1) -- (5.5,1) node[scale = 0.7,below] {$\Re \of{t} $};


\draw[fill] (1,1) circle [radius=0.05cm] node[below] {$T$};
\draw[fill] (3,1) circle [radius=0.05cm] node[below] {$T_1$};

\end{tikzpicture}
\caption{Time path}
\end{subfigure}
\begin{subfigure}[t]{.5\textwidth}  
\centering
\begin{tikzpicture}




\draw [domain=0:180, color=red, thick] plot ({0.7*cos(\x)-0.5}, {0.7*sin(\x)+1});
\draw [->, domain=0:120, color=red, thick] plot ({0.7*cos(\x)-3}, {0.7*sin(\x)+1});
\draw[color = red, thick] (-1.2,1.01) -- (-2.3,1.01);

\draw [domain=0:180, color=blue, thick] plot ({0.7*cos(\x)-0.5}, {-0.7*sin(\x)+1});
\draw [->, domain=0:120, color=blue, thick] plot ({0.7*cos(\x)-3}, {-0.7*sin(\x)+1});
\draw[color = blue, thick] (-1.2,0.99) -- (-2.3,0.99);

\draw[fill] (0.2,1) circle [radius=0.05cm] node[right] {$u_0$};

\draw[fill] (-0.5,1) circle [radius=0.05cm] node[below] {$u_+$};
\draw[fill] (-3,1) circle [radius=0.05cm] node[below] {$0$};

\end{tikzpicture}
\caption{Path in the fast unstable manifold}
\end{subfigure}
\caption{Time path and solution in the local complex tangent space} \label{fig:complex_time_path_continuation}
\end{figure}
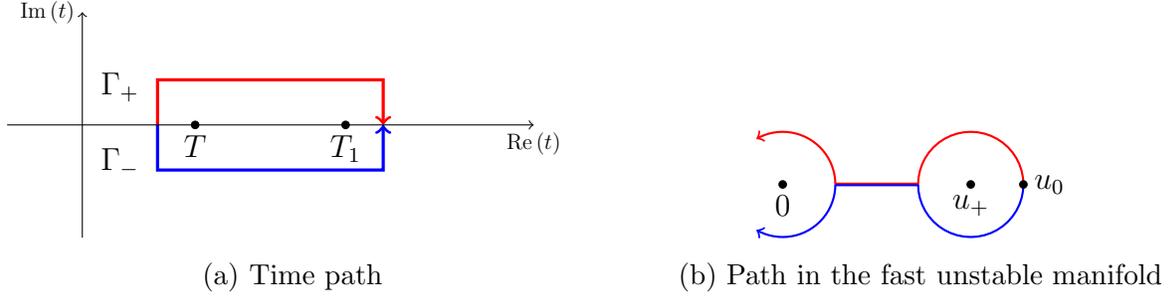

The main idea is depicted in Figure \ref{fig:complex_time_path_continuation}. Choosing the time paths $ \Gamma_\pm $, the path the solution $ u \of{\Gamma_\pm,x} $ traces in the unstable manifold is shown on the right. The argument, why the solution does not close up again, is that the eigenvalues at zero and $ u_+ $ are not the same. Suppose, that the local flow  around zero and $ u_+ $ is one-dimensional and purely linear with eigenvalues $ \mu_0 $ and $ \mu_1 $. To pass from $ u_0 $ to the other side of the equilibrium $ u_+ $ would then take time $ t_0 = \frac{i \pi}{\mu_1} $, whereas it would take time $ t_1 = \frac{i \pi}{\mu_0} $ at zero. Thus the solutions along $ \Gamma_\pm $ can only coincide if $ t_0 = t_1 + 2 n t_1 $ for some $ n \in \N $. 

\begin{lemma} \label{lem:eigenvalue_unstable_eq}
The eigenvalue $ \mu $ of the linearization at $ u_+ \of{x} $ satisfies $ \frac{\pi^2}{4} < \mu $.
\end{lemma}
\begin{proof}
We want to show that the largest eigenvalue of the operator
\[
L := \partial_{xx} + 2 u_+ \of{x} 
\]
is not equal to the absolute value of the first eigenvalue at zero. For that reason, we follow closely the discussion on Sturm-Liouville problems in \citep{brunt06}.
For Sturm - Liouville problems, we can express the largest eigenvalue by the Rayley quotient. 

\[
R \of{Q} := \frac{\int_{-1}^1 -Q_x^2 \of{x} + 2 u_+ \of{x} Q^2 \of{x} dx}{\norm{Q}_{L^2}^2}.
\]

By \citep{brunt06} the eigenvalue $ \mu $ satisfies the  following variational principle

\[
\mu := \max_{Q \in H^1} R \of{Q}.
\]



%
%

Thus, we can test with any function to obtain a lower bound of $ \mu $. Taking $ Q = u_+$ we obtain

\begin{align*}
&R \of{u_+} = \frac{\int_{-1}^1 -((u_+)_x)^2 \of{x} + 2 (u_+)^3 \of{x} dx}{\norm{u_+}_{L^2}^2} =  \frac{\int_{-1}^1 -((u_+)_x)^2 - 2 (u_+)_{xx} \of{x} (u_+) \of{x} dx}{\norm{u_+}_{L^2}^2} = \frac{\norm{ (u_+)_x}_{L^2}}{\norm{u_+}_{L^2}^2} \\
& \leq \mu.
\end{align*}

This yields the inequality

\begin{equation} \label{eq:mu_auxiliary_iq}
\norm{(u_+)_x}_{L^2}^2 \leq \mu \norm{u_+}_{L^2}^2.
\end{equation}

We can expand $ (u_+)_x \of{x} := \sum_{n=0}^\infty a_n e_n \of{x} $ where $ e_n \of{x} := \sin \of{\frac{n \pi}{2} \of{x+1}} $ is the eigenbasis at zero. By Plancherel we can rewrite inequality \eqref{eq:mu_auxiliary_iq} as 

\[
\sum_{n=1}^\infty \frac{n^2 \pi^2}{4} a_n^2 \leq \mu \sum_{n=1}^\infty a_n^2,
\]

for 
\[
u_+ = \sum_{n=1}^\infty a_n e_n,
\]

If, it holds that $ u_+ \of{x} \neq \sin \of{\frac{\pi}{2} \of{x+1}} $, also $  \frac{\pi^2}{4} < \mu $ holds. As simple calculation yields

\[
-\frac{\pi^2}{4} \sin \of{\frac{\pi}{2} \of{x+1}} + \sin \of{\frac{\pi}{2} \of{x+1}}^2 \neq 0,
\]
for $ x = 0 $.
\end{proof}

The previous lemma allows to show that the blow-up singularity is indeed a branch point of the analytic continued solutions. We summarize the results of the chapter in the following theorem, which was already quoted in the introduction.

\begin{theorem}
There exists a $ \delta > 0 $ such that the time analytic continuations of the real blow up orbit on the fast unstable manifold of $u_+ $, i.e. $ \Phi \of{t, (\tau, \Upsilon \of{\tau}} $, $ 0 < \tau < \delta $ exists and has the following properties:
\begin{enumerate}[label=(\roman*)]
\item It blows-up completely at time $ T$.
\item It can be continued to upper and lower spall strips $S_\pm \of{\delta, \left[ T, T_1 \right] }$ for 
\[
T_1 < 2 C_0 \max \left\{ \max_{x \in I} \frac{v_0 \of{x}}{\varphi \of{x}}, \norm{u_+}_\infty \right \}.
\]
The constant $ C_0 > 0 $ depends only on the heat semigroup.
\item The upper and lower time path continuations do not coincide after $ T_1 > 0 $
\end{enumerate}
\end{theorem}
\begin{proof}
The first claim is already known. We have already shown the second claim. 
The idea to prove the third claim is to choose a complex time path as indicated in the Figure \ref{fig:complex_time_path_continuation} and to show that the solution after continuation along path $\Gamma_+ $ and $ \Gamma_- $ does not coincide. 

Since the fast unstable manifold is analytic, we have can consider the reduced equation

\begin{equation} \label{eq:reduced_equation_fast_unstable}
\dot q = \mu q + \tilde f \of{q} , \qquad q \of{0} = q_0,
\end{equation}

where $ \tilde f \of{u} := P_{su} u \of{q, \Upsilon \of{q}}$ is an analytic function vanishing of quadratic order. Since $ \tilde f $ respects the real axis, real $ q_0 $ implies that the solution $ q \of{t} $ is real for real time $ t $. Take now $ q_0 > 0 $. The solution to the nonlinear heat equation with initial condition

\[
u_0 = u_+ + q_0 \varphi + \Upsilon \of{q_0}.
\]

blows-up.

The time needed to pass from the right side of $u_+ $ to the left side of $ u_+$ (see Figure \ref{fig:complex_time_path_continuation}) is $ t = \frac{i \pi }{\mu} $.
 
This is due to the Cauchy formula by separation of variables of the reduced equation \eqref{eq:reduced_equation_fast_unstable}. By Cauchy residue theorem, the total time needed to go around the stationary solution $ u_+ $ is

\[
2 t = \oint \frac{1}{\mu q + \tilde f \of{q}} dq = \frac{2 \pi i}{ \mu}.
\]
The factor two is due to conjugation symmetry. Furthermore the heteroclinic orbit converges to the first eigenfunction at zero, since the heteroclinic orbit does not change sign. 

Similarly as in the fast unstable manifold we prove that the time needed to pass from right to left around $ u_-$ in the first eigenfuction is close to

\[
t = \frac{4 i}{\pi}.
\]

But from Lemma \ref{lem:eigenvalue_unstable_eq} we know that  $ \frac{\pi^2}{4} < \mu $ which implies, that the solution is not real. In particular the solution is not real when going first for time $ \frac{i \pi }{\mu} $ and then along the real heteroclinic orbit for time $ t > T_1 $, such that the reduced equation on the slow stable manifold holds and then for time $ \frac{- i \pi }{\mu} $.
\end{proof}

\paragraph*{Upper blow-up rate estimate}

We prove a geometric characterization of blow-up points: If there is blow-up at $ t= 0 $ and $ x = 0 $, the image of the function $ \D \mapsto u \of{t+i, 0} $ must cover a half-plane in the complex plane, see Figure \ref{fig:complex_time_disk}. In particular, the image can not be contained in any sector of opening less than $\pi$. This surprising result follows from analytic functions theory. Since we already now, that the solution on the fast unstable manifold is contained in the upper half-plane by Lemma \ref{lem:orbit_bounded_hetero}, we also have an upper estimate on the blow-up rate for complex time. Note, that in contrast to lower estimates upper blow-up rate estimates are more difficult to obtain, and is not clear how to transfer the proofs e.g. in \citep{quittner07} from the real to the complex case. As already mentioned in the introduction, upper blow-up rate estimates allow for rescaled coordinates. In rescaled coordinates, the blow-up point becomes an equilibrium.

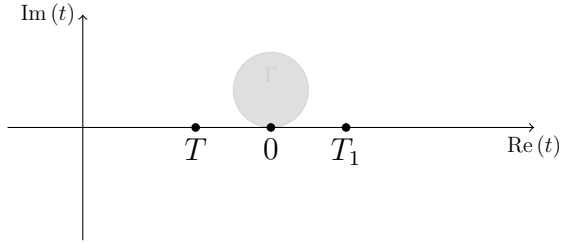
\begin{figure}[h]
\centering
\begin{tikzpicture}

%

%


\draw[fill, lightgray, opacity=0.5] (2,1.5) circle [radius=0.5cm] node[above, scale = 0.7] {$\Gamma$};


\draw[->] (-0.5,-0.5) -- (-0.5,2.5) node[scale = 0.7,left] {$\Im \of{t} $};

\draw[->] (-1.5,1) -- (5.5,1) node[scale = 0.7,below] {$\Re \of{t} $};


\draw[fill] (1,1) circle [radius=0.05cm] node[below] {$T$};
\draw[fill] (3,1) circle [radius=0.05cm] node[below] {$T_1$};
\draw[fill] (2,1) circle [radius=0.05cm] node[below] {$0$};

\end{tikzpicture}
\caption{Complex time disk attached to the blow-up point} \label{fig:complex_time_disk}
\end{figure}

\begin{lemma} \label{lem:upper_blow_rate_estimate}
Consider the solution constructed in Theorem \ref{thm:continuation_back_to_the_real_axis}. 
Then it holds $ \abs{u \of{t,x}} \leq \frac{M}{\Im \of{t}} $ for some $ M > 0 $ independent of $ x $. The constant $ M > 0 $ is related to the height $ \delta $  of the time strip in which the solution exists.
\end{lemma}
\begin{proof}
The image of $ u \of{S_+ \setminus \R,I} $ is contained in the upper half-plane of the complex plane by Lemma \ref{lem:orbit_bounded_hetero}.. The point evaluation $ \delta_x : C^0 \of{I} \to \C $
\[
\delta_x \of{u} := u \of{x},
\]
is a bounded linear functional. The function $ u^x \of{t} := \delta_x \of{ u \of{t} }$ is a holomorphic function. Take now a disk in the upper half-plane tangent to $ t^*$ of radius $ \delta/2$. First assume $ \delta = 2 $. Then the function $ u^x \of{t} $ is a holomorphic function from the unit disk to $ \C$. The proof follows from the classical subordination \citep{Duren_Hp_Spaces}  principle as follows
\begin{enumerate}[label=(\roman*)]
\item Set $ z = t - t^* - i $ and define the function
\[
\tilde u^x \of{z} := u^x \of{t^* + i + z} - u^x \of{t^* + i}.
\]
The function $ \tilde{u}^x $ is a holomorphic function on the unit disk at zero with $ \tilde u^x \of{0} = 0 $. Furthermore, the image $ \tilde{u}^x \of{\D} $ is contained in a shifted upper half-plane.
\item Next we look for a Möbius transformation $ \Gamma \of{z} := \frac{z + b}{c z + d} $, which satisfies
\[
\Gamma \of{0} = 0, \qquad \Gamma \of{\infty} = 1, \qquad \Gamma \of{i \alpha} = -1.
\]
for some $ \alpha \in \R $. The first two conditions yield $ b = 0 $ and $ c = 1 $. The last condition implies $ d = - 2 i \alpha $, thus
\[
\Gamma \of{z} = \frac{z}{z - 2 i \alpha}.
\]
The line $ t \mapsto i \alpha + t $, $ t \in \R $ is mapped to the unit circle since
\[
\abs{\Gamma \of{t + i \alpha}} = \abs{\frac{t + i \alpha}{t - i \alpha}} = 1, \qquad t \in \R.
\]
Now, we choose $ \alpha $ such that $ \tilde u^x \of{\D} \subset \Gamma^{-1} \of{\D} $. The inverse transformation is given by
\[
\Gamma^{-1} \of{z} = \frac{-2 i \alpha z}{1-z}. 
\]
\item Consider the function $ \omega \of{z} = \Gamma \circ \tilde u^x $. The function is a holomorphic self-map of the unit disk and satisfies $ \omega \of{0} = 0 $. Thus by Schwarz lemma $ \abs{\omega \of{z}} \leq \abs{z} $. This implies $ \tilde u^x \of{z} = \Gamma^{-1} \of{ \omega \of{z}} $ and thus 
\[
\sup_{\abs{z} \leq r } \abs{ \tilde u^x \of{z}} \leq \sup_{\abs{z} \leq r } \abs{ \Gamma_\alpha \of{ \omega \of{z}}} \leq \sup_{\abs{z} \leq r } \abs{ \Gamma_\alpha \of{z}} \leq \frac{2 \alpha}{1-r},
\]

for any $ r < 1 $.

\item Rescaling time, we can always rescale a disk of radius $ \delta / 2 $ to the unit disk. This gives the estimate
\[
\sup_{\abs{z} \leq r } \abs{ \tilde u^x \of{z}} \leq \frac{2 \alpha}{1-2 r \delta^{-1}}.
\]
for $ r < \delta/2 $.
\end{enumerate}
\end{proof}

\begin{remark}
This implies that the nontagential limit of $ u^x \of{t} (t-t^*) $ to $ t^*$ exists.
\end{remark}

\begin{remark}
The result is quite astonishing, since we do not just get an a priori estimate on the blow-up rate very easily, but we also get a geometric condition on the range of the image of points close to blow-up from the lower blow-up rate estimate. Suppose, that there exists an interval $ \tilde I \subset I $ and a disk $ \Omega $ in the upper complex plane touching the real axis such, that $ \im \of{u \of{\Omega,\tilde I}} $ is contained in a sector of the upper half-plane, then the solution is bounded on $ \tilde{I} $ up to the real axis. Then we get subordination of the solution by some $ p < 1 $ of the Cayley transform, which contradicts blow-up by Corollary \ref{cor:uniform_bounded_small}.
\end{remark}

\begin{figure}[h]
\centering
\begin{subfigure}[b]{.3\textwidth}
\centering
\begin{tikzpicture}[scale=1.5]

                
\draw[thick,domain=0:90,samples = 40,decoration={markings, mark=at position 0.75 with {\arrow[ultra thick]{>}}},
       postaction={decorate}] plot ({cos(\x)*sin(2*\x)}, {sin(\x)*sin(2*\x)});

\draw[thick,domain=0:90,samples = 40,decoration={markings, mark=at position 0.75 with {\arrow[ultra thick]{>}}},
       postaction={decorate}] plot ({1.5*cos(\x)*sin(2*\x)}, {1.5*sin(\x)*sin(2*\x)});
       
\draw[thick,domain=0:90,samples = 40,decoration={markings, mark=at position 0.75 with {\arrow[ultra thick]{>}}},
       postaction={decorate}] plot ({cos(\x)*sin(2*\x)}, {-sin(\x)*sin(2*\x)});

\draw[thick,domain=0:90,samples = 40,decoration={markings, mark=at position 0.75 with {\arrow[ultra thick]{>}}},
       postaction={decorate}] plot ({1.5*cos(\x)*sin(2*\x)}, {-1.5*sin(\x)*sin(2*\x)});

\draw[thick,domain=0:90,samples = 40,decoration={markings, mark=at position 0.75 with {\arrow[ultra thick]{>}}},
       postaction={decorate}] plot ({-cos(\x)*sin(2*\x)}, {-sin(\x)*sin(2*\x)});

\draw[thick,domain=0:90,samples = 40,decoration={markings, mark=at position 0.75 with {\arrow[ultra thick]{>}}},
       postaction={decorate}] plot ({-1.5*cos(\x)*sin(2*\x)}, {-1.5*sin(\x)*sin(2*\x)});
       
\draw[thick,domain=0:90,samples = 40,decoration={markings, mark=at position 0.75 with {\arrow[ultra thick]{>}}},
       postaction={decorate}] plot ({-cos(\x)*sin(2*\x)}, {sin(\x)*sin(2*\x)});

\draw[thick,domain=0:90,samples = 40,decoration={markings, mark=at position 0.75 with {\arrow[ultra thick]{>}}},
       postaction={decorate}] plot ({-1.5*cos(\x)*sin(2*\x)}, {1.5*sin(\x)*sin(2*\x)});

\draw[thick,decoration={markings, mark=at position 0.75 with {\arrow[ultra thick]{>}}},
        postaction={decorate}] (0,0) -- (1,0);
\draw[thick,decoration={markings, mark=at position 0.75 with {\arrow[ultra thick]{>}}},
        postaction={decorate}] (0,0) -- (-1,0);

\draw[thick,decoration={markings, mark=at position 0.75 with {\arrow[ultra thick]{<}}},
        postaction={decorate}] (0,0) -- (0,1);
\draw[thick,decoration={markings, mark=at position 0.75 with {\arrow[ultra thick]{<}}},
        postaction={decorate}] (0,0) -- (0,-1);

\end{tikzpicture}
\caption{Solution for $p=3$}
\end{subfigure}%
\begin{subfigure}[b]{.3\textwidth}
\centering
\begin{tikzpicture}[scale=1.5]

\draw[thick,thick,domain=0:60,samples = 40,decoration={markings, mark=at position 0.5 with {\arrow[ultra thick]{>}}},
       postaction={decorate}] plot ({cos(\x)*sin(3*\x)}, {sin(\x)*sin(3*\x)});

\draw[thick,thick,domain=0:60,samples = 40,decoration={markings, mark=at position 0.5 with {\arrow[ultra thick]{>}}},
       postaction={decorate}] plot ({1.5*cos(\x)*sin(3*\x)}, {1.5*sin(\x)*sin(3*\x)});
       
\draw[thick,thick,domain=0:60,samples = 40,decoration={markings, mark=at position 0.5 with {\arrow[ultra thick]{>}}},
       postaction={decorate}] plot ({cos(\x)*sin(3*\x)}, {-sin(\x)*sin(3*\x)});

\draw[thick,thick,domain=0:60,samples = 40,decoration={markings, mark=at position 0.5 with {\arrow[ultra thick]{>}}},
       postaction={decorate}] plot ({1.5*cos(\x)*sin(3*\x)}, {-1.5*sin(\x)*sin(3*\x)});

\draw[thick,thick,domain=0:60,samples = 40,decoration={markings, mark=at position 0.5 with {\arrow[ultra thick]{<}}},
       postaction={decorate}] plot ({-cos(\x)*sin(3*\x)}, {-sin(\x)*sin(3*\x)});

\draw[thick,thick,domain=0:60,samples = 40,decoration={markings, mark=at position 0.5 with {\arrow[ultra thick]{<}}},
       postaction={decorate}] plot ({-1.5*cos(\x)*sin(3*\x)}, {-1.5*sin(\x)*sin(3*\x)});
       
\draw[thick,thick,domain=0:60,samples = 40,decoration={markings, mark=at position 0.5 with {\arrow[ultra thick]{<}}},
       postaction={decorate}] plot ({-cos(\x)*sin(3*\x)}, {sin(\x)*sin(3*\x)});

\draw[thick,thick,domain=0:60,samples = 40,decoration={markings, mark=at position 0.5 with {\arrow[ultra thick]{<}}},
       postaction={decorate}] plot ({-1.5*cos(\x)*sin(3*\x)}, {1.5*sin(\x)*sin(3*\x)});

\draw[thick,thick,domain=60:120,samples = 40,decoration={markings, mark=at position 0.5 with {\arrow[ultra thick]{>}}},
       postaction={decorate}] plot ({cos(\x)*sin(3*\x)}, {sin(\x)*sin(3*\x)});

\draw[thick,thick,domain=60:120,samples = 40,decoration={markings, mark=at position 0.5 with {\arrow[ultra thick]{>}}},
       postaction={decorate}] plot ({1.5*cos(\x)*sin(3*\x)}, {1.5*sin(\x)*sin(3*\x)});

\draw[thick,thick,domain=60:120,samples = 40,decoration={markings, mark=at position 0.5 with {\arrow[ultra thick]{>}}},
       postaction={decorate}] plot ({cos(\x)*sin(3*\x)}, {-sin(\x)*sin(3*\x)});

\draw[thick,thick,domain=60:120,samples = 40,decoration={markings, mark=at position 0.5 with {\arrow[ultra thick]{>}}},
       postaction={decorate}] plot ({1.5*cos(\x)*sin(3*\x)}, {-1.5*sin(\x)*sin(3*\x)});

\draw[thick,decoration={markings, mark=at position 0.75 with {\arrow[ultra thick]{>}}},
        postaction={decorate}] (0,0) -- (1,0);
\draw[thick,decoration={markings, mark=at position 0.75 with {\arrow[ultra thick]{<}}},
        postaction={decorate}] (0,0) -- (-1,0);

\draw[thick,decoration={markings, mark=at position 0.75 with {\arrow[ultra thick]{>}}},
        postaction={decorate}] (0,0) -- (-1.3*0.5,-1.3*0.86);
\draw[thick,decoration={markings, mark=at position 0.75 with {\arrow[ultra thick]{<}}},
        postaction={decorate}] (0,0) -- (1.3*0.5,1.3*0.86);

\draw[thick,decoration={markings, mark=at position 0.75 with {\arrow[ultra thick]{<}}},
        postaction={decorate}] (0,0) -- (1.3*0.5,-1.3*0.86);
\draw[thick,decoration={markings, mark=at position 0.75 with {\arrow[ultra thick]{>}}},
        postaction={decorate}] (0,0) -- (-1.3*0.5,1.3*0.86);

\end{tikzpicture}
\caption{Solution for $p=4$}
\end{subfigure}%
\caption{Real time ODE invariant solutions for different $\dot x = x^p $} \label{fig:invariant_regions_1}
\end{figure}
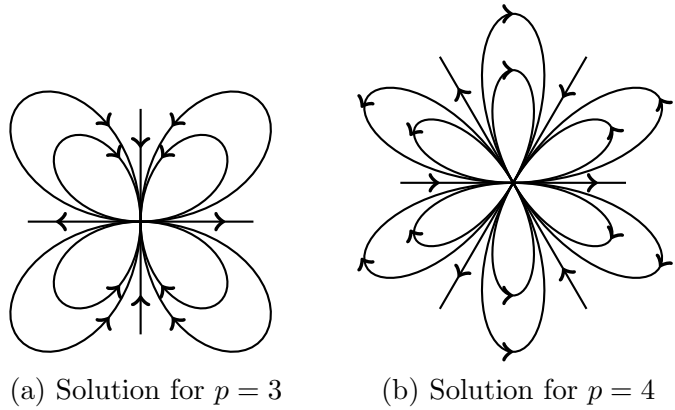

\begin{remark}
We expect that the proofs of this paper can be generalized for a general $ p > 1 $. The main reason is, that also in the case $ p > 1 $ there exists invariant regions of the ODE flow for positive solutions, see Figure \ref{fig:invariant_regions_1}. 
\end{remark}

\newpage

\section{Discussion}

We have shown that heteroclinic orbits introduce unbounded solutions and in particular blow-up solutions in the complex domain.
 \\
Furthermore, we studied the concrete example of the quadratic nonlinear heat equation. Here, we were able to derive finer results. For example there exists an analytic continuation of the blow-up orbit to a complex time strip cut out a finite time interval even though the blow-up is complete. Furthermore we showed that the analytic continuation along different paths around the blow-up point does not close. 

Quite recently a paper \citep{cho2016} with numerical simulations of the quadratic nonlinear heat equation with periodic boundary conditions in complex time appeared. The numerical simulations suggest, similar to our rigorous proven results in the Dirichlet case, that complex time continuation after the blow-up introduces a Riemann surface and in particular continuations along different paths around the singularity do not close up. Even though the analysis is purely numerical it indicates how continued solutions look like and give a starting point of a more detailed study of continued solutions. Especially a better estimate of the resurrection time $ T_1 $ is important, since we expect blow-up to only happen at a single point on the complex time plane. \\
A second direction is to study the connection of bounded and blow-up solutions in other settings. This is in particular interesting for the Navier-Stokes equation, where the problem of blow-up is still an open problem.

%
%



\bibliography{bibliography.bib}

\end{document}